\documentclass[reqno,10pt]{amsart}
\usepackage{fullpage}
\usepackage{aageneral}
\usepackage{aamath}
\usepackage{mhequ}
\usepackage{aachem}

\usepackage{microtype}
\usepackage{float,flafter}  %to properly position the table
\usepackage[hidelinks]{hyperref}
\usepackage{amssymb,amsthm,amsmath}
\usepackage{dsfont}   %for the 1-function
\usepackage{url}      %for url code in bibtex
\usepackage{mathtools}%for extendable arrows like xrightleftharpoons
\usepackage{color}
\usepackage{enumerate} %to alter list bullets
\usepackage{todonotes}

\usepackage[nocompress]{cite}

\def\rp#1{{\color{blue} #1}}

\newcommand{\AC}{\mathcal{AC}}

\newcommand{\cadlag}{c\`adl\`ag }

\newcommand{\e}{\mathrm{e}}

\newcommand{\I}{\mathcal{I}}

\newcommand{\J}{\mathcal I^{\mathrm{bd}}}
\mathchardef\mhyphen="2D % Define a "math hyphen"
\renewcommand{\emptyset}{\varnothing}

\newcommand{\NN}{\mathbb{N}}

\newcommand{\PP}{\mathbb{P}}

\newcommand{\RR}{\ensuremath{\mathbb{R}}}
\renewcommand{\S}{\mathcal{S}}

\newcommand{\II}{\mathds{1}}

\newcommand{\dd}{\mathrm{d}}
\newcommand{\norm}[1]{\ensuremath{\left\lVert{#1}\right\rVert}}
\renewcommand{\abs}[1]{\ensuremath{\left\lvert{#1}\right\rvert}}

\newcommand{\one}{\mathds{1}}

\newcommand{\rr}{z}

\newcommand{\rrr}{\mathfrak{z}}

\newcommand{\dpp}{\delta^{\prime\prime}/2}

\def\epsilon{\varepsilon}

\newcommand{\Acal}  {{\mathcal A}}
\newcommand{\Bcal}  {{\mathcal B}}
\newcommand{\Ccal}   {{\mathcal C }}

\newcommand{\Ecal}   {{\mathcal E }}

\newcommand{\Gcal}   {{\mathcal G }}
\newcommand{\Hcal}   {{\mathcal H }}
\newcommand{\Ical}   {{\mathcal I }}
\newcommand{\Jcal}   {{\mathcal J }}
\newcommand{\Kcal}   {{\mathcal K }}

\newcommand{\Rcal}   {{\mathcal R }}
\newcommand{\Scal}   {{\mathcal S }}

\newcommand{\BBp}{\mathrm{B}}

\DeclareMathOperator\Lip{Lip}

\newtheorem{thm}{Theorem}
\makeatletter
  \newcommand{\eqnum}{\leavevmode\hfill\refstepcounter{align}\textup{\tagform@{\thealign}}}
\makeatother

\definecolor{darkspringgreen}{rgb}{0.09, 0.45, 0.27}
\def\nv{n^{v}}
  \def\nvp{n_+^v}
\def\tnv{\tilde n^{v}}
\renewcommand{\d}{\mathrm{d}}
\renewcommand{\SS}{\mathcal S}
\DeclareMathAlphabet{\mathbbold}{U}{bbold}{m}{n}

\author{Andrea Agazzi}
\address{Department of Mathematics, Duke University, 120 Science Dr, Durham, NC 27710, USA}
\email{agazzi@math.duke.edu}

\author{Luisa Andreis}
\author{Robert I. A. Patterson}
\author{D. R. Michiel Renger}
\address{Weierstraß-Institut f\"ur Angewandte Analysis und Stochastik, Mohrenstrasse 39, 10117 Berlin, Germany}
\email{\{andreis, patterson, renger\}@wias-berlin.de}

\title{
\large Large deviations for Markov jump processes\\with uniformly diminishing rates%\\\small Escaping the boundary in autocatalytic chemical reactions
}
\date{\today}

\begin{document}

\begin{abstract} We prove a large-deviation principle (\abbr{LDP}) for the sample paths of jump Markov processes in the small noise limit when, possibly, \emph{all} the jump rates
  vanish \emph{uniformly}, but slowly enough, in a region of the state space. We further discuss the optimality of our assumptions on the decay of the jump rates. As a direct application of this work we relax the assumptions needed for the application of \abbr{LDP}s to, \emph{e.g.}, Chemical Reaction Network dynamics, where vanishing reaction rates arise naturally particularly the context of mass action kinetics.
\end{abstract}
\maketitle

\section{Introduction}

\subsection{Large deviations of Markov jump processes}
\label{subsec:intro ldps}

We study a family of $d$-dimensional Markov jump processes $\{X^{v}\}_{v \in \mathbb N}$ with state space $(v^{-1} \mathbb{Z})^d$, deterministic initial condition $X^{v}(0) = x_0^v \in (v^{-1} \mathbb{Z})^d$ and generator: %acting on suitable test functions  $f:\RR^d\to\RR$,
\begin{equation}\label{e:generator}
  (\LL^v f)(x):=v\sum_{r\in\mathcal{R}}\Lambda^{v}_r(x)\big\lbrack f(x+v^{-1}{\gamma^r}) - f(x)\big\rbrack.
\end{equation}
Here $\mathcal{R}$ is the finite set of possible jumps, $\gamma^r\in\mathbb{Z}^d$ are the fixed jump vectors, and $v\Lambda^{v}_r~:~(v^{-1}\mathbb{Z})^d \to [0,\infty)$ the associated jump rates. The parameter $v$ controls the noise in the system, and the scaling is chosen so that $\Lambda^{v}_r(x)$ converge as $v\to\infty$. Under this scaling it is known that the paths $X^v$ concentrate on solutions of the \emph{fluid limit} ODE~\cite{Kurtz1970}:
\begin{equation}\label{eq:hydro_lim}
\frac{\d}{\d t} x(t) =\sum_{r\in\mathcal{R}}\lambda_r(x(t))\gamma^r\,,
\end{equation}
where the continuous rates $\lambda_r~:~\mathbb R^d \to [0,\infty)$ are the limits of $\Lambda^{v}_r(x)$ as $v\to\infty$.

The process~\eqref{e:generator} and the ODE~\eqref{eq:hydro_lim} are used as microscopic and macroscopic models for a wide range of applications. For example, in the context of chemical reactions $X^v$ denotes the concentrations (number of molecules per unit volume) of $d$ species being transformed by a set of reactions $\R$. Here, for each reaction $r\in \mathcal R$, the vectors $\gamma^r$ encode which species are removed and created when a reaction $r$ occurs, and $\lambda_r(x)$ are the reaction rates~\cite{Kurtz1972}. In that setting $v$ can be interpreted as the volume size over which the concentrations are averaged. Other typical applications using similar models include biological systems involving predator-prey interaction, birth/cell division and death, biological fitness models, as well as epidemiological models.
%Under reasonable assumptions, this will be the principal situation we aim to model in this paper, but we present notation in a more general setting to accommodate more general classes of models.
%\mrnote{I don't think we need references here, but let me know if we do!}
In these settings, large-deviation techniques are often applied to simplify the dynamical landscape of the complex, high dimensional microscopic model~\eqref{e:generator} while retrieving quantitative information about the random fluctuations around the mean~\eqref{eq:hydro_lim}. This information can be used for example to study non-equilibrium thermodynamics~\cite{esposito19,MielkePattersonPeletierRenger2015}, to speed up simulations of rare events~\cite{delmoral04, vde04, vde19}, or to study spontaneous transitions between metastable states \cite{freidlin12} also in the multi-scale setting \cite{popovic19}.

%\aanote{expand this paragraph}
The classical proof of the Large-Deviation Principle (\abbr{LDP}) uses a tilting, also called a change-of-measure technique. The main challenge there is that the tilting can only be performed around sufficiently regular paths, whereas the large-deviation principle needs to be proven for any non-typical path. Therefore, the large-deviation lower bound requires an approximation argument, either for the random process or for the rate functional.
A particular challenge in either case is to approximate a path without changing its starting point, which is required when proving the large-deviation principle under a deterministic initial condition.
This becomes more difficult if the jump rates vanish in some regions of the state space, which is however an inherent property of the models used in many application domains. For example, in the context of chemical kinetics, it is natural to assume that the rate of a chemical reaction vanishes when the concentration of one or more of the reactants approaches $0$. Similarly in the context of infectious disease models, the rate of spread of a virus is usually modelled as a linear function of the infected population.

% \subsection{Escaping the boundary}
% \label{subsec:autocatalysis} \label{s:12}

\begin{example}\label{ex:ex1}
  The problem is nicely illustrated by the simplest model for autocatalysis or cell division, in chemical notation: $\mathsf{A}\to 2\mathsf{A}$. In this case there is only one species and one jump, so we may write the linear jump rate as $\Lambda^{v}(x)=\lambda(x)=x$, starting from a concentration with one particle, i.e. $X^v(0)=x_0^v=1/v$. Clearly, the process converges to the solution of $\dot x(t)=x(t)$ with initial condition $x(0)=x_0=0$, that is $x(t)\equiv0$. In other words, the process is expected to stick to the degenerate set $\partial \mathcal S = \{x=0\}$, which corresponds to the boundary of the state space $\mathbb R_+$. However, it can be calculated (as an application of~\cite[Lemma 2.1]{Kordecki97}) that $\lim_{v \to \infty} v^{-1}\log\PP\big(X^v(t)\geq\delta\big)=\delta\log(1-e^{-t})$ for any fixed time $t>0$ and $\delta>0$.
%, see Appendix~\ref{app:autocatalysis}.
%\lanote{Will we add the appendix, in the end?}\aanote{I think it would be sufficient to cite Rob's paper, it is an easy calculation once one knows that}
Although this is a large-deviation result about the marginal $X^v(t)$, it suggests that the paths can also ``escape'' from the boundary with finite large-deviation cost.
On the other hand, we note that choosing $x_0^v\equiv0$ implies that $X^v(t)\equiv0$ almost surely. Hence, whether or not escape is possible depends strongly on the initial condition. This is a similar principle as what is sometimes called a ``well-prepared initial condition'' in $\Gamma$-convergence theory~\cite{Mielke2016}.
\end{example}

For more general models of the type~\eqref{e:generator}, one expects that as long as the jump rates $\lambda_r(x)$ {do not vanish too rapidly approaching the degenerate points} and when starting at a well-chosen initial condition near such points, then the process {will} be able to escape with finite large-deviation cost, and the large-deviation principle should still hold. In this paper, we show that this is indeed the case. More specifically, denoting by $\partial \mathcal S$ the set where some -- or possibly \emph{all} -- of the reaction rates $\lambda$ vanish, we prove that when the rates near {$\partial \Scal$} decrease slower than $\exp\pc{ -1/\mathrm{dist}(x,\partial\Scal)^{\alpha}}$, \ie
\begin{equation}
    \mathrm{dist}(x,\partial\Scal)^\alpha \log \lambda(x) \to 0 \qquad \text{as }\quad \mathrm{dist}(x,\partial\Scal) \to 0,
    %\sim \exp\Big( -\frac{1}{\mathrm{dist}(x,\partial\Scal)^\alpha}\Big).
\label{eq:Kratz Pardoux rate}
\end{equation}
the large-deviation principle holds under the assumption $\alpha\in\lbrack0,1)$. Furthermore, we show that these conditions on the decay of the rates close to the degenerate set are, in some sense, necessary and not just sufficient.

% Our approach completely bypasses the change-of-measure technique by estimating directly on the process level. The core of the approach consists of showing separate estimates on the total number of jumps and on the types of jumps. This allows us to relax known assumptions to \eqref{eq:Kratz Pardoux rate} with $\alpha\in\lbrack0,1)$.

%by approximating the jump process by a Poisson process with piecewise constant coefficients, or by approximating the nontypical limit path by analytic regularisation.

%For every (small) time interval, a tilting of the measure then gives the desired local rate function. Shrinking the size of the approximating pieces yields a Riemann sum, that in the limit becomes the well known integral in the rate function. This method can be applied whenever the rates are bounded away from 0, but when this is not the case the method collapses, as the change of measure is not well-defined in that case.

\subsection{Literature and approach}

Early papers~\cite{Feng1994dynamic,Leonard1995} establishing a sample-path large-deviation principle for jump Markov processes mimicked the Dawson--G\"artner approach~\cite{Dawson1987}, where one first derives an abstract large-deviation result for the empirical measure on paths, and then contracts it to obtain a large-deviation principle for the path of the empirical measure. Another approach, now considered the classic tilting technique, was first used in~\cite{ShwartzWeiss95} assuming that all the jump rates are uniformly bounded away from zero in the domain of interest. In~\cite{ShwartzWeiss05}, the authors relaxed this condition by assuming the existence \emph{a subset} of jumps with rates that are uniformly bounded away from zero and that ``push'' the process away from the degeneracy region. This assumption is further relaxed in \cite{dupuis16,agazzi18,agazzi182,anderson18}, requiring only the existence of a \emph{sequence} of jumps that sequentially transport the process away from the problematic region.

Recent works have taken steps to generalise these assumptions to the uniformly vanishing case (when all rates can vanish in some region of state space), in the context of chemical kinetics \cite{pattersonrenger19} and in the context of infectious disease models \cite{pardoux16,pardoux17}. These papers give sufficient conditions on the models at hand to bypass the technical difficulties encountered in the proof of the \abbr{LDP} lower bound when some of the jump rates are not bounded away from zero. We mention that the work \cite{pattersonrenger19} assumes ``sufficiently random'' initial conditions to bypass the problem, but we shall focus on a deterministic initial condition.

The problem of vanishing rates is addressed more completely in \cite{pardoux16,pardoux17}, where the authors obtain large-deviation estimates for vanishing rates when the microscopic initial condition allows escaping the degenerate set with positive -- although vanishing in $v$ -- probability. Their approach is based on a {careful} adaptation of the standard tilting argument to obtain the \abbr{LDP} lower bound for processes. In particular, to bypass the problem of jump rates vanishing in some regions of state space, the change of measure performed by the authors depends on the large-deviation scaling parameter $v$, which is inversely proportional to the jump size. This replaces the problem of escaping to an $\mathcal O(1)$ distance from {these degeneracy regions} uniformly in $v$ to the one of escaping to $\mathcal O(1/\sqrt v)$. Their result allows for jump rates that behave as in~\eqref{eq:Kratz Pardoux rate} for $\alpha\in\lbrack0,1/2)$.

In this paper we bypass the change-of-measure technique altogether establishing more direct and concise bounds on the process level.  The main challenge in realizing this strategy is to identify a set of paths occurring with sufficient probability to recover the \abbr{LDP} lower bound while allowing for simple estimation of such probability. The core of the approach consists of showing separate estimates on the total number of jumps and on the types of jumps for paths in such set. This allows us to extend the assumption~\eqref{eq:Kratz Pardoux rate} to any $\alpha\in\lbrack0,1)$ while covering a larger family of processes than the existing literature. In addition, we provide a counterexample showing that our upper bound for the exponent $\alpha$ is optimal: If the rates of the process decay as \eref{eq:Kratz Pardoux rate} for $\alpha \geq 1$ with sufficient uniformity, as we make precise below, the process will no longer be able to escape the degenerate set with finite large-deviation cost.

\subsection{Outline} The paper is structured as follows. In Section~\ref{Sec:not_res} we introduce our notation, we list our assumptions and we state our main result, Theorem~\ref{th:main}, namely the \abbr{LDP}. We also illustrate the generality of our result with some examples. The proof of Theorem~\ref{th:main} is split into two sections: In Section~\ref{Sec:lower} we prove the \abbr{LDP} lower bound, while Section~\ref{Sec:upper} deals with the \abbr{LDP} upper bound. Finally, in Section~\ref{Sec:counterex}, we discuss the optimality of our assumptions on the decay of the rates and in Section~\ref{Sec:constants} we summarize the key quantities of the proof.

\section{Notation and results}\label{Sec:not_res}

We start by giving a concrete example of the properties of  systems we aim to generalize in this paper, introducing some important quantities in an intuitive way:
\begin{example}[Mass action kinetics] \label{ex:massaction}
  %The above remark bounds the decay of the limiting jump rates as the boundary is approached.
  %An example of reaction rates satisfying \eref{e:decay} and therefore \aref{a:escape} c) is given by asymptotically polynomial rates such as the one encountered in the context of mass action kinetics.
  In the context of chemical kinetics, one indexes the dimension of state space with a set of \emph{species} $\{\mathsf{S}_i\}$ representing the chemical compounds in the system of interest. To describe interactions between different compounds one defines reactions $r \in \mathcal R$ via $\gamma^{r,\mathrm{in}}, \gamma^{r,\mathrm{out}} \in \mathbb N_0^d$ and
\begin{equation*}
  \sum_{i=1}^d\gamma^{r,\mathrm{in}}_i \,\mathsf{S}_i\longrightarrow \sum_{i=1}^d\gamma^{r,\mathrm{out}}_i\,\mathsf{S}_i.
\end{equation*}
This is to be understood as saying $\gamma^{r,\mathrm{in}}_i$ copies of species $\mathsf{S}_i$ are consumed in each $r$-reaction while $\gamma^{r,\mathrm{out}}_i$ copies are produced.
The reaction rate $\Lambda_r^v$ is specified via a rate constant $k_r\geq 0$ as
%\footnote{Throughout the paper we use $\norm{x} = \sum_i \abs{x_i}$ because this is more natural than the Euclidean norm when vectors are lists of concentrations. \mr{But here we use $\lVert\cdot\rVert_1$?}} \mr{(I corrected these jump rates of the chemical master equation:)}
  \begin{equ}
    \Lambda_r^{v}(x) := k_r \frac{1}{v^{\sum_i|\gamma_i^{r,\mathrm{in}}|}}\prod_{i=1}^d \binom{vx_i}{\gamma^{r,\mathrm{in}}_i}\gamma^{r,\mathrm{in}}_i!% \binom{vx}{vx - \gamma^{r,\mathrm{in}}_i}
    \qquad
    \forall x \in \pc{v^{-1}\mathbb N_0}^d\,,
  \end{equ}
%where $\|\cdot\|_1$ denotes the $\ell^1$ norm.
where $\binom{\cdot}{\cdot}$ denotes the binomial coefficient. The jump vector is $\gamma^r:=\gamma^{r,\mathrm{out}}-\gamma^{r,\mathrm{in}}$.
 These rates are bounded from above on compact sets, and they converge   to $\lambda_r(x) = k_r \prod_{i=1}^d x_i^{\gamma_i^{r,\mathrm{in}}}$ as $v\to\infty$. It is easy to see that $\Lambda_r^v(x) = 0$ whenever $v x_i < \gamma^{r,\mathrm{in}}$, so that $X^v \in (v^{-1}\mathbb N_0)^d$ almost surely. Consequently, the degenerate set of the limiting dynamics is a subset of the boundary $\partial \mathbb R_{\geq 0}^d$. As we shall illustrate in examples below, different choices of reactions  result in $X^v$ being confined on subsets of $(v^{-1}\mathbb N_0)^d$. \end{example}

%note{I think the most natural order is: define a sequence $x_0^v$, define the corresponding $S$ and then define its limit.}
We start by defining the set of reachable points of the process. Throughout, we fix a sequence of deterministic initial conditions $\{x_0^v\}$. By the potentially degenerate character of the stochastic dynamics at hand, we reduce the state space $(v^{-1}\mathbb Z)^d$ to the set of reachable points of the process with that initial condition $x_0^v \in (v^{-1}\mathbb Z)^d$:
\begin{equ}
\mathcal{S}_v :=\{x\in  \pc{v^{-1}\mathbb Z}^d\colon \mathbb{P}[\exists t\geq0,\, X^{v}(t)=x \mid X^{v}(0)=x_0^v]>0\}\,.
\end{equ}
Note that, by definition, $\Lambda^{v}_r(x) =0$ whenever $x + v^{-1} \gamma^r \notin \Scal_v$ for any $x \in \mathcal S_v$.
Assuming that in the limit $v \to \infty$, the initial values $x_0^v \in \pc{v^{-1}\mathbb Z}^d$ converge to $x_0\in \mathbb R^d$, we
 write $\displaystyle{\mathcal{S}=\bigcap_{n\in\mathbb{N}}\overline{\bigcup_{v\geq n}\mathcal{S}_v}}$ where the raised line indicates topological closure.
We assume throughout that $\S$ is compact, but discuss how to relax this assumption in \rref{r:exptight}.
 %Without loss of generality, we consider jump rates which are zero on the boundary of $\Scal$. Therefore, w
 We associate to $\mathcal{S}$ the set of jumps
\begin{equation*}
\mathcal{R}_{\geq 0}:= \left\{ r \in \Rcal \colon  \exists x\in \mathcal{S},~\lambda_r(x)>0\,\right\}
\end{equation*}
%where $\S^o$ is the interior of $\SS$.
Notice that, depending on the sequence of initial conditions, the same Markov process may have a different state space $\SS_v$ and different set of jumps $\Rcal_{\geq 0}$. We refer to \exref{ex:ex2} for a situation where $\R_{\geq 0}\neq \R$. However, by abuse of notation, we will drop the index $\geq 0$ and refer to this set simply as $\Rcal$.

Finally, we define the \emph{degenerate} set -- also referred to as ``boundary'' from its topological characterization in many application domains  -- as $\partial \S := \{x \in \S~:~ \exists r \in \Rcal,~ \lambda_r(x)=0\}$. This represents the set of points where the limiting process  is degenerate, \ie where the classical proof of the large-deviation principle will not immediately apply. Observe that this is a slight abuse of both notation and terminology, since this degenerate set $\partial\S$ may be different from the actual topological boundary of the set $\S$.

% The second example shows some of the subtleties that are addressed via the definition of $\SS$.
%
% \begin{example}\label{ex:ex1}
% Consider the following autocatalytic reaction model (which can also be seen as a model for binary cell division) \mrnote{Need to connect to Sect. 1.2.}
% \begin{align}\label{e:mainexample}
%   \mathsf{A} \to 2\mathsf{A}.
% \end{align}
% with mass action kinetics.
% The reaction rate for the single reaction in this process is linear and satisfies $\Lambda^{v}(x) =x$ for all $x\geq 1/v$ (and consequently $\lambda(x)=x$ for all $x\in\RR_{\geq0}$).
% The generator is
% \begin{align*}
%   (\LL^v\phi)(x):=v x\big\lbrack \phi(x+1/v) - \phi(x)\big\rbrack.
% \end{align*}
% and the macroscopic dynamics are $x(t) = x(0) e^t$.
% The large deviations of this process can be calculated directly, but after bounding the jump rate at some large concentration, we are able to verify that the rate function is $\int_0^T \left[\dot x(t) \log  \dot x(t) -  \dot x(t) + 1\right] \d t$ even for $x_0=0$ with well-prepared initial condition $X^v(0) := 1/v$.
% \end{example}

The following example clarifies the role of the  sequence of initial conditions on the resulting state space.
{\begin{example}\label{ex:ex3}
The mass action kinetics model
$
  \mathsf{A} \leftrightarrow \mathsf{B}$ (see \exref{ex:massaction} for definition of the rates and jump vectors)
with initial conditions $x_0^v = (0, 1+1/v)$ results in $\mathcal S_v = \{x \in (v^{-1} \mathbb Z_{\geq 0})^2~:~x_1+x_2 = 1+1/v\}$ and $\mathcal S = \{x \in \mathbb R_{\geq 0}^2~:~x_1+x_2 = 1\}$.
\end{example}}
{\begin{example}\label{ex:ex2}
The nontrivial effect of different sequences of initial conditions is captured by the system
\begin{align*}
  \mathsf{B} \leftrightarrow 2\mathsf{B} \qquad \mathsf{A} \leftrightarrow \mathsf{2A + B},
\end{align*}
with mass action kinetics (see \exref{ex:massaction}).
For this model, the sequence $x_0^v = (1/v,0)$ results in $\mathcal S_v = (v^{-1} \mathbb Z_{\geq 0})^2 \setminus \{0\}$ and $\SS = \mathbb{R}_{\geq 0}^2$.
However, if $x_0^v = (0,1/v)$ we have $\mathcal R_{\geq 0} = \{B \to 2B, 2B \to B\}$ and the dynamics are restricted to $\mathcal S_v = \{0\}\times (v^{-1} \mathbb Z_{\geq 0}) $ resulting in $\SS = \{x\in \mathbb R_{\geq 0}^2~:~x_1 = 0\}$.
%and the rate function of any path leaving this set will be infinity.
\end{example}}

%The main purpose of this paper is to adapt the proof to this setting.
%\aanote{I understand we want to be as general as possible, but perhaps we can assume that all rates satisfy this?}

%{The set $\SS_*$ can be interpreted as the (limit of the) space that is reachable with positive probability by the jumping process for a given sequence of initial conditions $x_0^v$. }
%\red{comment: we do not really want the rates to be non-zero in every point, right? It may be that they are either $0$ or bounded from below+they can be zero in point inside $\Scal_*^\circ$ as in Kratz-pardoux.}
%\paragraph{Some preliminaries}

%Let $k\in \mathbb{R}_+^{\Rcal}$ and $\xi\in \Scal$

\subsection{\bf Assumptions} \label{s:assumptions}
To ensure existence of the limit, we require the reaction rates to satisfy some conditions.
\begin{assumption}[Convergence and regularity of rates]\label{a:rate-converg}
We assume the following.
\begin{enumerate}[a)]
% \item  $\Lambda^{v}_r(x) =0$ whenever $x +\frac1v \gamma^r \notin \Scal_v$.\lanote{This is true by DEFINITION, am I right?}\aanote{I agree, we should take it out, I have added a comment after the definition of Sv}
\item  There exists a collection of non-negative functions $\{\lambda_r\}_{r\in\Rcal}$, {Lipschitz} continuous on a neighborhood of $\Scal$ in $\mathbb R^d$, such that
\begin{equation*}
\lim_{v \to \infty} \sup_{x \in \Scal_v}\sum_{r\in \Rcal}\abs{{\Lambda_r^{v}(x)} - \lambda_r(x)} = 0.
\end{equation*}
%\item \aa{$\exists \,C<\infty$ such that ${ \sup_{x \in \Scal} \lambda_r(x)} < C$ for all $r \in \R$.}\aanote{should go since compact and contnuous rates?}
\item {There exists $\aleph > 0$ such that for all $r\in\Rcal$, $v>0$ and $x \in \S_v$} with $\Lambda^{v}_r(x)>0$, we have
\begin{equ}
  \frac {\Lambda_r^{v}(x)}{\lambda_r(x)} \geq \aleph.
\end{equ}
\end{enumerate}
\end{assumption}

% Now we may define conditions on the sequence $\{X^{v}\}$ such that it satisfies a \abbr{LDP}.\aanote{can this go?}
% \begin{assumption}\label{a:escape}
% There exists $\Scal_*\subseteq \Scal$ such that for all $v \in  \mathbb N$ we have  $\Scal_{X^{v}(0)} \supseteq \Scal_*$ and $\lim_{v \to \infty} X^{v}(0) = x_0 \in \Scal_*$.
% \end{assumption}
% \begin{assumption}\label{a:initial}
% $\PP\left(X^v(0) \in S\right) = 1$ and the distributions of the $X^v(0)$ satisfy an LDP with rate $v$.
% \end{assumption}

% \begin{assumption}[Convergence of the initial conditions]
%
% \end{assumption}

As we outline in Section~\ref{s:path} the proof of our main theorem is based on the construction of short linear paths moving the process away from the boundary $\partial \Scal$. We now introduce notation to decompose the state space into subsets, in each of which the linear path will be fixed.
More precisely, following a standard approach first presented in \cite{ShwartzWeiss05}, we cover the state space $\S$ with {the relative interior of} finitely many convex, compact sets $\{\Acal_i\}_{i \in \I}$  with {$\Acal_i \subseteq \S$} for all $i \in \I$. We then define  $\partial \mathcal A_i := \partial \Scal \cap \Acal_i$ and let $\J \subseteq \I$ be the subset of indices for which $\partial \Acal_i  \neq \emptyset$.
%\aanote{FOR LUISA: changed slightly this sentence, the old version is in green below, if you agree with the change please erase note and old.}\aa{for which there exists {a boundary point} $x $ s.t. $\lambda_r(x) = 0$ for some $r \in \R$ and denote by $\partial \Acal_j$ the subset of such $x$.}

We now present the assumptions for the lower bound. We assume that, whenever the process starts from an initial condition close to $\partial \Scal$ (where possibly all the rates are zero), one can identify a finite sequence of favorable jumps, which we call the \emph{escape sequence}, that push the process away from the boundary. We further crucially assume that the rates of such favorable jumps do not decay too fast as we approach the boundary.
%, while other reactions are not too strong in this region.
This is captured by the following counterexample.
{
%Finally, we present an example to support the thesis that our assumptions on the decay of the relevant jump rates when approaching the degeneracy region is, in some sense, optimal.
\begin{example}\label{ex:diverging1d}
Consider the family of Markov jump processes $\{X^v\}$ with generator
%\footnote{Note that while the state space is not compact in this and some of the following examples, exponential tightness holds by uniform Lipschitz continuity of the rates.}
\begin{equation}\label{eq:example}
 \LL^v f(x) := v e^{-\frac k {x}}\pc{f(x+v^{-1})-f(x)}\qquad \text{for } f~:~v^{-1}\Nn_0 \to \Rr\,,
\end{equation}
 for any $k >0$. The above process, which for small $x$ is a time-changed version of the autocatalytic process introduced in \exref{ex:ex1}, has only one possible jump in the positive direction with rate $\Lambda^{v}(x) =  e^{-\frac k {x}}$ s.t.  $  \lim_{\rho \to 0}\rho \pc{\inf_{x \colon x\geq \rho } \log\lambda(x)} = -k \neq 0\,. $
 For the sequence of initial conditions $x_0^v = 1/v$, we have $\Scal_v = v^{-1} \mathbb N$ and $\Scal = \RR_{\geq0}$.  Then, for any $w>0$ and $\epsilon \in (0,w/2)$ the probability of observing a realization of $X^v$ in an $\epsilon$-neighborhood of the path $z(s) = s w$ on the interval $s \in [0,1]$ can be trivially estimated as
 \begin{equ}
   \mathbb P\left[\sup_{t\in [0,1]}|X_t^v - z(t)|\leq \epsilon \right] \leq \mathbb P[X_1^v \geq w- \epsilon ]\leq \mathbb P[X_1^v \geq \widetilde w ]
   \end{equ}
   for $\widetilde w = \min(k, w-\epsilon)/2$. Denoting by $\tau_i$ the waiting time between the $i-1$-th and $i$-th jump of the Poisson process $X_t^v$ at $x \in \mathcal S_v$, we further have
 \begin{equation*}
     \mathbb P[X_1^v \geq \widetilde w ]
     \leq \prod_{i = 1}^{\lfloor v \widetilde w\rfloor} \mathbb P[\tau_i \leq 1 ]
     = \prod_{i = 1}^{\lfloor v\widetilde w\rfloor} 1-\exp[-(ve^{-kv/i})]
     \leq \exp\left(\sum_{i=1}^{\lfloor v\widetilde w\rfloor}\left(\log v -kv/i\right)\right). %\leq  2^v \exp\pq{v\pc{\log (v) - k \sum_{i = 1}^{\lfloor v(wT- \epsilon)\rfloor}\frac 1 i}}
\end{equation*}
The rough estimate above yields
 \begin{equation}
   \frac 1 v \log \mathbb P[X_1^v \geq \widetilde w ] \leq
    \frac 1 v \lfloor v\widetilde w\rfloor \log v - k \sum_{i = 1}^{\lfloor v\widetilde w\rfloor} \frac1i
    < (\widetilde w - k) \log v - k (1+\log \widetilde w)
 \end{equation}
which approaches $-\infty$ as $v \to \infty$.

%In particular the above system does not satisfy \aref{a:escape}b).
% {\begin{proposition}\label{p:counterexample} %Assume that all assumptions except \aref{a:escape} ii) hold.
% For any $k > 0$ there exists $r^*$ s.t. for $X^v_0 = v^{-1}$ the rate function for any path $\rr(t)\in AC(0,1; \RR_{\geq0})$ with $r(1) \geq r^*$ is infinite.
% %Then there exists a choice of $r,k$ such that the process in \exref{ex:ex1} does not obey a \abbr{ldp} at rate $v$.
% \end{proposition}
% }
% \la{ This is not a proof that this system does not satisfy a \abbr{LDP} with rate $v$, it gives however a good hint on the reason why we believe that  \aref{a:escape} a) is optimal.}
% {The above example shows that the rate of decay assumed in \aref{a:escape} a) is optimal to obtain a \abbr{ldp} with rate $v$: a path with finite rate function when \aref{a:escape} a) holds loses this property on the borderline case of the assumption.}
\end{example}
As the example above shows, sufficiently fast decay of the rates of the process $X^v$ implies the divergence of the large-deviation cost of any nontrivial path starting on the boundary $\partial \Scal$. We now proceed to give sufficient assumptions guaranteeing that this does not happen in the general setting. In particular, to capture the idea of escaping a boundary in the higher dimensional setting, for each $\mathcal A_i$ we define directions $w_i$ with some structural properties (\aref{a:escape} a)) allowing to construct linear paths that leave such boundaries. We assume that these paths can be realized as a sequence of jumps $\mathcal E_i$ whose rates do not decay too fast (\aref{a:escape} b)-c)), as to avoid for the realization of such path to have an infinite large-deviation cost.
}
Denoting throughout by $\BB_\rho(x)$ the Euclidean ball of radius $\rho$ in $\Rr^d$  and by $|A|$ the Lebesgue measure of the set $A$, we summarize such assumptions below:

 \begin{assumption}[Escape]\label{a:escape}
There exist constants $\epsilon, \epsilon', \epsilon'' > 0$ such that for each $j \in \I$, the following holds:\\
\begin{enumerate}[a)]
\item If $j \in \Ical^{bd}$ there is a $w_j \in \Rr^d$  with $\|w_j\|=1$ and $\kappa_j\in(0,1)$ such that whenever $x \in \mathcal A_j$ and $\inf_{ y \in \partial \Scal}\|x-y\| < \epsilon'$
%\rp{Should this be distance from $\partial \Acal_j$, I do not think it can be true as stated (I was the last editor!)} \la{I agree, ``whenever $\inf_{y \in \partial \Acal_j}\|x-y\| < \epsilon'$''}
and $t \in (0,\epsilon)$
\begin{enumerate}[i)]
\item $t \mapsto \inf_{y\in \partial \Scal}\|x+t\,w_j-y\|$ is increasing, and
\item $ \BB_{t \kappa_j}(x+t\,w_j) \cap \partial \Scal = \emptyset$.
\end{enumerate}
We write $\kappa_- = \min_{j\in \I^{bd}} \{\kappa_j\}$. If $j \in \Ical \setminus \Ical^{bd}$ we choose $w_j = 0$. \label{assit:wj}

\item There exists a finite sequence $\Ecal_j:=(r_{1}, \dots, r_{n_j})$ of jumps in $\Rcal$ with
\begin{equation*}
 {\limsup_{v\to \infty} \frac 1 v \left|{\log \Lambda_{r_{k}}^{v}\pc{x_0^v + v^{-1}\sum_{i = 1}^{k-1} \gamma^{r_{i}}} }\right| = 0,}
 \qquad k = 1, \dotsc n_j\,,
\end{equation*}
and
$\sum_{i=1}^n \gamma^{r_{i}} = \alpha_j w_j$ for some $\alpha_j >0$\,. \label{assit:Ej}
%such that for all $t \in (0,\epsilon)$ and all $x \in \partial A_j$ \rp{or even $ \Acal_j\cup \bigcup_v Z_0^v$?} $x+tw_j \notin \partial \Scal$.

\item Defining  $Z_0^v := \{x_0^v + v^{-1}\sum_{i = 1}^{k} \gamma^{r_{i}}~:~k \in (0, \dots, n_j-1)\}$, for all $r \in \Ecal_j$ and $T>0$
%letting $x_s := x_0 + s\, w_j$  %and $\tau = \inf\{t>0~:~x_s \not \in \BB_j\}$
%the function $ \log\lambda_r(x_s)$ is \emph{uniformly} on $x_0 \in  \Acal_j\cup \bigcup_v Z_0^v$ integrable on $[0,\rho]$ for $\rho$ small enough, \ie
\begin{equation*}
  \lim_{\rho \to 0} \sup_{x \in  \Acal_j\cup \bigcup_v Z_0^v} \int_0^{\rho}\!\left|\log \lambda_r(x + sw_j)\right|\d s = 0\,.
\end{equation*}
\label{assit:escape rate}
% \begin{equation*}
%   \lim_{\delta \to 0} \sup_{x \in  \Acal_j\cup \bigcup_v Z_0^v} \sup_{t \leq T-\delta}\int_t^{t+\delta} \left|\log \lambda_r(x + sw_j)\right| \d s = 0\,.
% \end{equation*}
% \rp{I think we need to restrict to the integral $\int_0^\delta$, because, for large $s$, $x+sw_j$ may hit another boundary or even leave $\Scal$.}\la{Yes, it makes sense to me to have
%
% even if we could just ask that this is true only when $x + sw_j$ is in the interior $\forall s\in (t,t+\delta)$. Was there a reason to have the general $t$? To take into account shifts?
% }

\item Let $\mathcal W_{j,\kappa''} := \{w_j+y~:~\|y\|< \kappa''\}$. There exists $\kappa''<\kappa_-/3$ such that for any $x \in \mathcal A_j \cup \bigcup_v Z_0^v$ we have that for all $r \in \R$ with
%$\lim_{\rho \to 0} \inf_{x \in \Acal_j \colon d(x,\partial B_j)> \rho} \lambda_r(x) = 0$
$\lambda_r(x) < \epsilon''$
the rates $\lambda_r(\,\cdot\,)$ are nondecreasing {along paths $t \mapsto x + t w$  for any $w \in \mathcal W_{j,\kappa''}$, for $t\in  (0,\epsilon)$. } %\mr{Doesn't this imply \ref{assit:wj})i)? Or is the strict monotonicity essential?}
\label{assit:monotonicity}
%\blue{RP: Do we want $Z_0^v$ to move with $x$ or to be fixed to $x_0^v$?}
\end{enumerate}
%Furthermore, there exists $\kappa_j,\beta_j>0$ and $\epsilon >0$ such that for all $t \in (0,\epsilon)$ and $l$
%\begin{equation}
 %   \lambda_{r_l^j}\left(x^v+tw_j +\frac1v \sum_{k=1}^{l-1}\gamma^{r^j_k}\right) \geq \kappa_j t^{\beta_j}.
%\end{equation}

\end{assumption}
\noindent It is readily verified \aref{a:rate-converg} and \ref{a:escape} are satisfied by mass action kinetics rates on a convex domain \cite{agazzi18}.

{%The above assumption is necessary to ensure that a escape sequence exists and that its large-deviation cost is finite.
\begin{remark}\label{r:decay} While \aref{a:escape} c) is natural in terms of our proof, we note that it is automatically satisfied whenever there exists $\alpha \in [0,1)$ such that
  \begin{equ}\label{e:decay}
    \lim_{\rho \to 0}\rho^\alpha \pc{\inf_{x \in \Acal_j {\cup \bigcup_v Z_0^v} \colon d(x,\partial \Acal_j)> \rho} \log\lambda_r(x)} \to 0 \quad \text{for all } r \in \mathcal E_j\,,
  \end{equ}
  as we mentioned in \eqref{eq:Kratz Pardoux rate}, where $d(A,B) := \inf_{x \in A, y \in B} \|x-y\|$.
  In particular, the above decay condition implies the results of \cite{ShwartzWeiss05} and \cite{pardoux16}. These papers make the stronger assumptions that  \eref{e:decay} holds with $\alpha = 0$ and $\alpha \in [0,1/2)$ respectively.
\end{remark}}

\subsection{The large-deviation principle}
\label{subsec:the ldp}

{For a parameter $T>0$ fixed throughout the paper,} we denote by $D_u(0,T;\RR^d)$ (resp. $D_s(0,T;\RR^d)$) the space of \cadlag functions with values in $\RR^d$ endowed with the topology of uniform convergence (resp. Skorohod topology).
Furthermore we define $\BBp_{[0,T]}(\rho,z)$ to be the ball of radius $\rho$ in $D_u(0,T;\RR^d)$. Finally, for $z\colon [0,T]\mapsto \RR^d$ in the set $\Acal \Ccal(0,T;\RR^d)$ of absolutely continuous functions, we denote by $\dot z$ its time derivative and we will say that $z\in \Acal \Ccal(0,T;\Scal)$ whenever $z(t)\in \Scal$ a.e. $t \in [0,T]$.

%\begin{definition}[LDP] We say that a \emph{Large Deviations Principle} with rate function $I(\,\cdot\,)$ holds if for all measurable $\Gamma \subseteq D_s(0,T; \RR^d)$
%\end{definition}

To define the standard rate function for the \abbr{LDP} of Markov jump processes in the small noise limit \cite{ShwartzWeiss95} we introduce the action
 %\aanote{do we need the dual formulation here?}
\begin{equs}
  I_{[0,T]}(z) & := \begin{cases}\int_0^T \inf_{ \pg{\mu\in\RR^\R :~ \sum_{r \in \R} \mu_r \gamma^r = \dot z}}\! \Hcal( \mu|  \lambda(z(t)))\,\d t,
  &\text{ if }z\in  \Acal \Ccal(0,T;\Scal)~,\\
  +\infty&\text{ otherwise. }
  \end{cases} \label{e:lbrf} \\
              % & = \int_0^T\!\Big\lbrack\sup_{\xi\in\RR^d} \xi\cdot x'(t) - \sum_{r\in\R}\lambda_r(x_t)\big(e^{\xi\cdot \gamma^r}-1\big)\Big\rbrack\,dt,\label{e:dual form I}\\
   \Hcal( \mu| \lambda) &:= \sum_{r \in \R} \lambda_r - \mu_r + \mu_r \log \frac {\mu_r}{\lambda_r}  \label{e:lbrf2}\,,
 \end{equs}
  where $( \lambda(x))_r := \lambda_r(x)$.
We can now state the main result of this paper:

\begin{thm}\label{th:main}
Consider the sequence of Markov jump processes $\{X^{v}\}_{v\in \mathbb{N}}$, fixing a  sequence of deterministic initial conditions $\{x_0^v\}_{v\in \mathbb{N}}^\infty$ with $x_0^v\in (v^{-1} \mathbb Z)^d$ such that  $X^v(0)=x_0^v \to x_0 \in  \Rr^d$.
%and let $\mathcal S_v, \mathcal S$ be as defined in \eqref.
Furthermore, let \aref{a:rate-converg} and~\ref{a:escape} hold.
%then if $X^{v}(0)\rightarrow x\in \Scal$ we say that
Then the sequence $\{X^v\}_{v\in \mathbb{N}}$ satisfies a \abbr{LDP} in $D_u(0,T;\RR^d)$ (resp. $D_s(0,T;\RR^d)$) with good rate function
\begin{equ}
  I_{[0,T]}^{x_0}(z)  := \begin{cases}I_{[0,T]}(z)
  &\text{ if }z(0) = x_0\\
  +\infty&\text{ otherwise,}
  \end{cases}  \label{e:rf2}
 \end{equ}
that is, $I_{[0,T]}^{x_0}$ has compact sublevel sets $D_u(0,T;\RR^d)$ (resp. $D_s(0,T;\RR^d)$), and for all measurable $\Gamma \subseteq D_s(0,T; \RR^d)$,
\begin{align}
  \limsup_{v \to \infty} \frac 1 v \log\px{x_0^v}{X^v \in \Gamma} &\leq - \inf_{x \in \bar \Gamma} I_{[0,T]}^{x_0}(x) \label{eq:LDP UB}\\
  \liminf_{v \to \infty} \frac 1 v \log\px{x_0^v}{X^v \in \Gamma} &\geq - \inf_{x \in \Gamma^o}    I_{[0,T]}^{x_0}(x), \label{eq:LDP LB}
\end{align}
where $\px{x_0^v}{\,\cdot\,}$ denotes the conditional probability on $X_0^v = x_0^v$.
\end{thm}
% \rp{The second part of the definition is true, but obvious, since this is an event that has zero probability for every (large enough) $v$.  I think it is included in the first case since leaving a closed set in an impossible direction would require a positive time interval when the flux was not absolutely continuous with respect to the jump rates.}\aanote{I think this is standard to be expressed like this, because for example in the upper bound formulation you have a derivetive, which does not make sense when the curve is not AC}

The compactness of $\S$ and Lipschitz continuity of the rates $\{\lambda_r\}$ directly implies that the rates are bounded and Lipschitz, which in turn implies that the process is exponentially tight. Therefore, the compactness of sublevel sets of $I_{[0,T]}^{x_0}$ comes for free, and the upper bound only needs to be proven for compact sets \cite[Lem.~1.2.18]{Dembo1998}.
%\mr{I wrote some lines about exponential tightness here, since it is an essential part of our proof (albeit standard), and should not belong in a remark. However we seem to refer to that remark a couple of times, so maybe the reference is now wrong. Alternatively we write this as a short lemma and refer to that?}

  %\lanote{Here uniform Lipschitz continuity is a consequence of Lipschitz continuity given by assumption 1a) and the fact that $\Scal$ is compact.}
%  This gives the exponential tightness of the process, which in turn allows to extend a local \abbr{ldp} to its full form as in \tref{th:main}, significantly simplifying the proof of the upper bound.

A few considerations are now in order.
\begin{remark}\label{r:exptight}
%  The compactness assumption on the space $\S$ directly implies, by Lipschitz continuity of the rates $\{\lambda_r\}$ their boundedness and uniform Lipschitz continuity.
  %\lanote{Here uniform Lipschitz continuity is a consequence of Lipschitz continuity given by assumption 1a) and the fact that $\Scal$ is compact.}
%  This gives the exponential tightness of the process, which in turn allows to extend a local \abbr{ldp} to its full form as in \tref{th:main}, significantly simplifying the proof of the upper bound.
  The boundedness of the rates and compactness of $\S$ can be relaxed. Indeed, exponential tightness can be obtained by other means: Either by Lipschitz continuity of the jump rates \cite{feng2006,ShwartzWeiss95}, or by stability estimates \cite{agazzi18, agazzi20}. Once exponential tightness is guaranteed, one can restrict the analysis to trajectories that do not leave a large enough compact \cite[Theorem 4.4]{feng2006}, effectively reducing the problem to the one with compact state space, which we discuss above.
  % we can prove the upper bound only with an  adaptation of classical results  for jump Markov processes in \eg \cite{dupuis91, ShwartzWeiss95}, which requires the use of \aref{a:escape}-b). Since we want to stress what happens to the upper bound when \aref{a:escape}~b) does not hold, we decided to have a slightly stronger assumption on $\Scal$, namely that it is compact, in order to relax conditions on the decay of the rates.
  %
  % In the setting of \aref{a:rate-converg} and \ref{a:escape}, let us suppose  that for all $j \in \mathcal I$ there is a constant $K_j^+ > 0$ such that  we have ${ \sup_{x \in \Acal_j} \log\lambda_r(x)} < K_j^+$ for all $r \in \R$. Then, we can relax the condition on $\Scal$ being compact. The boundedness of the rates gives exponential tightness \lanote{add comment on Feng\&Kurtz  theorem 4.4 implying exponential tightness here}  and in this case the \abbr{LDP} still holds.
\end{remark}
We further note that the seemingly restrictive assumption of \emph{deterministic} initial condition also covers the case when such an initial condition is random. This can be done, given the probability conditioned on a fixed initial state from \tref{th:main}, by integrating with respect to the probability distribution $\nu^{v} \in \mathcal M((v^{-1}\mathbb Z)^d)$ of the initial condition, provided that the measure $\nu^{v}$ satisfies some weak regularity and tightness assumptions.
In this case, however, one must check that the conditions in \tref{th:main} hold \emph{uniformly} on a set of positive measure \abbr{wrt} $\nu^{v}$. For a detailed discussion of this procedure when $\nu$ also satisfies a \abbr{LDP} at the same rate we refer to \cite{biggins04}.

% Escaping the boundary means that after a small time $\tau>0$ one observes macroscopically a non-trivial $c_\tau=\delta>0$. We proceed to establish exponential upper and lower bounds for the number of jumps in the interval $[0,\tau]$.

% \subsection{Structure of the paper} {The paper is structured as follows: in section 2 we give a sketch of the proof of the main result of the paper, in section 3 we establish large deviations bounds on escape probabilities for time-marginals of the jump process and transfer such result to pathwise estimates, proving \tref{th:main}. In section 4 we prove the upper bound and discuss in which sense \aref{a:escape}a) is optimal.}

 \section{Proof of \abbr{LDP} lower bound}\label{Sec:lower}\label{s:lb}
 % This section is dedicated to the proof of the following result:
 % \begin{proposition}[LDP lower bound]
 %   Consider the sequence of Markov jump processes $\{X^{v}\}$, fix a  sequence of initial conditions $\{x_0^v\}_{v=1}^\infty$ with $X^v(0)=x_0^v \to x_0 \in  \Rr_{\geq 0}^d$ and let $\mathcal S_v, \mathcal S$ be as in \eqref{eq:S_v}. Furthermore, let \aref{a:rate-converg} and~\ref{a:escape} hold.
 %   %then if $X^{v}(0)\rightarrow x\in \Scal$ we say that
 %   Then the sequence $\{X^v(\cdot)\}_{v = 1}^\infty$  is such that, for any open set $\Gamma$ in $D_u(0,T;\Scal)$ (resp. $D_s(0,T;\Scal)$) it holds
 %   \begin{equ}
 %     \liminf_{v \to \infty} \px{x}{X^v \in \Gamma} \geq - \inf_{\rr \in \Gamma^o} I_x(\rr)
 %   \end{equ}
 %   where $I_x(\cdot)$ is defined in \eref{e:rf2}
 % \end{proposition}\aanote{notation for rate function: $I_x$ or $I_{[0,T]}$?}

 %We denote throughout by $\BBp_{[0,T]}(\delta, \rr)$ the ball of radius $\delta$ around the path $\rr$ in the supremum norm.

The general strategy adopted to prove the \abbr{LDP} lower bound result is mainly standard. Without loss of generality, we may assume that $\Gamma$ is open, for any path $\rr \in \Gamma$ one can find a $\delta>0$ such that $\BBp_{[0,T]}(\delta, \rr)\subset \Gamma$, so that
$\px{x}{X^v \in \Gamma} \geq \px{x}{\BBp_{[0,T]}(\delta, \rr)}$ for $\delta >0$ small enough. Hence, it is sufficient to prove that, for any path $\rr \in \Gamma$ the probability that the process $X^{v}$ stays in a neighborhood $\BBp_{[0,T]}(\delta, \rr)$ for any $\delta>0$ is approximately $\exp[-v I^{x_0}_{\lbrack0,T\rbrack}(\rr)]$.
Applying such estimate to a sequence $\{\rr^{(n)}\}_{n =1}^{\infty}$ of paths converging to the minimizer of $I_{[0,T]}^{x_0}$ in $\Gamma$ with small enough $\delta^{(n)}$ proves the desired result. This shows that for the lower bound~\eqref{eq:LDP LB} it is sufficient to prove the following.
\begin{proposition}\label{l:28}  Fix  a path $\rr~:~[0,T]\to \mathcal S$ with a fixed initial condition $\rr(0) = x_0 \in \Scal$
%, which takes values in $\Acal_j$ for all $t \in [0,T]$
%\la{(for $i\in\Ical\setminus \Jcal$? Or at least for $x$ such that $\lambda_r(x)\geq c>0$ uniformly bounded from below, right?)}
%and
such that {$I^{x_0}_{\lbrack0,T\rbrack}(\rr ) = K < \infty$}. Then, for a sequence of initial conditions $x_0^v \in (v^{-1}\mathbb Z)^d$ converging to $x_0$, under \aref{a:rate-converg} and \aref{a:escape},
\begin{align}\label{e:28}
\lim_{\delta \rightarrow 0}\liminf_{v \rightarrow \infty} \frac{1}{v} \log \px{x_0^v}{X^v \in \BBp_{[0,T]}(\delta, \rr)} \geq - I_{[0,T]}^{x_0}(\rr)\,.
\end{align}
\end{proposition}

The remainder of this section concentrates on proving such estimate. Our approach to the proof of the above result mimicks the one from \cite{ShwartzWeiss05}. Throughout this section, we fix a path $\rr \in \AC([0,T], \SS)$ starting from $\rr(0)=x_0$ and we approximate $\rr$ with another path $\rrr_\delta$ obtained by perturbing $\rr$, shifting it uniformly away from the regions where the rates are degenerate by a quantity controlled by $\delta$. We then proceed to prove on one hand that
the probability of $X^v$ approximately following $\rrr_\delta$ is accurately described by the rate function $I^{x_0}_{[0,T]}(\rrr_\delta)$, and on the other that the large-deviation cost of the process following the \emph{shifted} path converges towards the one of the original path as $\delta \to 0$.
{The main difficulty to establish the former claim arises with the necessity of keeping the microscopic initial condition of the path fixed, and estimating the probability of the process reaching, in a small time interval, the origin of the shifted path $\rrr$, which is \emph{macroscopically} bounded away from the boundary. On the other hand, to establish the latter convergence property of the rate functional we have to guarantee sufficient regularity of such functional as some of the jump rates decrease to $0$ with $\delta \to 0$.
The remainder of the section is devoted to the realization of this program. In \sref{s:path} we give the explicit construction of the path $\rrr_\delta$ and detail its role in the proof of \pref{l:28}, in \sref{s:jumpbounds} we estimate the probability of the process reaching the origin of the shifted path from its fixed initial condition $x_0$, while in \sref{s:rates} we prove sufficient regularity of the rate functional $I_{\lbrack0,T\rbrack}^{x_0}$. The proof of \pref{l:28} is finally concluded in \sref{s:patching}}.

 \subsection{Construction of the path $\rrr_\delta$}\label{s:path}

 \begin{figure}[t]
  \centering
  \def\svgwidth{.65\textwidth}
  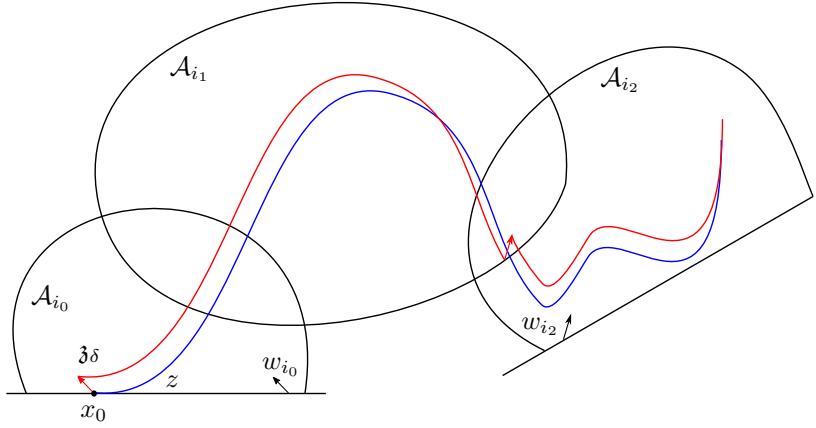
  \caption{Schematic representation of shifted path.}
  \label{f:scaling}
 \end{figure}

We now construct a macroscopic path that perturbs the original path $\rr \in \mathcal A\C(0,T; \S)$ at a negligible cost and that can only intersect $\partial \mathcal S$ at its initial point. {To do so we recall the covering $\{\mathcal A_i\}$ defined in \sref{s:assumptions}, allowing us to identify, for each $\Acal_i$, directions $w_i$ to move away from the boundary $\partial \mathcal S$. More specifically, this covering allows to partition the path $\rr$ as it enters different regions $\Acal_i$ and to shift it in the corresponding direction $w_i$, thereby guaranteeing that the \emph{shifted} path avoids $\partial \mathcal S$ as we detail below and  as depicted in \fref{f:scaling}.}

To construct $\rrr_\delta$ we introduce the sequence of times $\{\tau_k\}_{k}$ so that $\{\rr(t)~:~t \in [\tau_{k},\tau_{k+1}]\}\subset \Acal_{i_k}$ for all $k$ for a corresponding sequence $\{i_k\}_k$ of indices in $\Ical$.
Then, for fixed $x_0 \in \Scal$, we consider $i_0\in\Ical$ such that $x_0\in\Acal_{i_0}$, so that $\rr(t)\in \Acal_{i_0}$ for all $t \in [0,\tau_1]$. In this interval we define the shifted path
 % {For any $\rr \in \mathcal{AC}(0,T,\Rr^d)$ with initial condition $\rr(0) = x_0 \in\Acal_j$ and some $w_j\in \mathbb R^d$ the new path is defined as}
  \begin{align}
    \rrr_\delta(t) := \begin{cases}  x_0 +  t w_j & \text{for } t \in [0, t_\delta] \\
     {\rr(t - t_\delta) - x_0 + t_\delta w_j} & \text{for }t \in ( t_\delta, \tau_1+t_\delta] \end{cases}\label{e:rtilde}
  \end{align}
  for $t_\delta = \frac 1 6\xi \min(\delta , \omega_\rr^{-1}(\delta))$  where $\omega_\rr$ denotes the modulus of continuity of the path $z$ and $\xi >0$ is defined below (see \lref{l:breakup}).
  %and {for $y$ in a small neighborhood of $x+t_\delta w_j$ that we define below.}
  We then continue defining the path $\rrr_\delta$
  {by shifting the original path $\rr$ infinitesimally on every interval $[\tau_k, \tau_{k+1}]$, sequentially moving it away from $\partial \Acal_{i_k}$ with the corresponding $w_{i_k}$.  More precisely, setting the length of the $k$-th shift time for the perturbed path as $\beta^k t_\delta$ for $\beta>1$ to be chosen later (see \lref{l:breakup}) and denoting by $\Delta_k^\beta := t_\delta \sum_{\ell=0}^{k-1} \beta^\ell $ the cumulative shift time up to transition $k$
  %starting with $\rrr_\delta(0) = x_0 \in \Acal_{i_0}$, we let $\rrr_\delta(s) := x_0 + s w_{i_0}$ for $s \in [0,t_\delta]$. We further define $\rrr_\delta(s) := t_\delta w_{i_0} + \rr(\tau_1)$ on $s \in [t_\delta, \tau_1 + t_\delta]$. Then,
  % assuming that the path has entered $\Acal_{i'}$,
   we define the $i$-th shift as
    \begin{equation}\label{e:rtilde2}
      \rrr_\delta(s) :=
      \begin{cases}
        \rrr_\delta(t)+ (s - \tau_i - \Delta_{i}^\beta) w_{i} \qquad &\text{for } s \in (\tau_i + \Delta_i^\beta, \tau_i + \Delta_{i+1}^\beta]\\
        \rrr_\delta(\tau_i+ \Delta_i^\beta)+\beta^k t_\delta w_{i} +  \rr(s- \Delta_{i+1}^\beta
        %-\Delta_{i}^\beta
        ) - \rr(\tau_i
        % - \Delta_{i+1}^\beta- \Delta_{i}^\beta
        ) &  \text{for } s \in (  \tau_i+\Delta_{i+1}^\beta, \tau_{i+1} + \Delta_{i+1}^\beta].
      \end{cases}
    \end{equation}
}

  %By a slight abuse of notation, we denote $\trr_\delta$ the path defined in~\eqref{e:rtilde} when it is continuous, i.e. $y=x_0+t_\delta w_j$.

%{what do we need the above for?}

  %\paragraph{ \bf Sketch of the construction and main lemma.}\la{Shall we keep this? }

  %By the local nature of \abbr{ldp} lower bounds we would like to show that, the cost of any path $\rr(t)$ (in particular any path with $r(0) = 0$) that microscopically starts \emph{away} from $0$ is given by the rate function \eref{e:lbrf}.
We now establish some structural properties of the newly constructed path around the original $\rr$, which we recall is fixed throughout this section. This lemma extends \cite[Lemma 3.4]{ShwartzWeiss05}.

\begin{lemma}\label{l:breakup}
{Let Assumptions~\ref{a:rate-converg} and \ref{a:escape} hold and set $\beta := 3/\kappa''$ recalling that $\kappa''<\kappa_- = \min_{j \in \Ical} \kappa_j$ from \aref{a:escape} a).}  Then,
for any
%$T > 0$ ,
$K > 0$ there is a $J>0$ such that if $I_{[0,T]}(\rr) \leq K$, there are $0 = \tau_0 < \tau_1 <
\dots < \tau_J = T$ and $\{i_k\}$ with $\rr(t) \in \Acal_{i_k}$ for $\tau_{k-1} \leq t \leq \tau_k$.
Furthermore, setting $\xi:= \min(1,(\kappa''/3)^{J+1}/3, \epsilon)$ there exists
%$\xi_\rr\in \left(0,\min_j(1/6, \epsilon_j)\right)$ and
 $\delta_\rr>0$ such that for all
%$\xi< \xi_\rr$ \mrnote{Is this the same $\xi$ as the constant to be chosen later from \eqref{e:rtilde}?} for all
$\delta<\delta_\rr$ the path $\rrr_\delta$ from \eref{e:rtilde} and \eqref{e:rtilde2} satisfies
$\sup_{[0,T]} \|\rr-\rrr_\delta\| <  2\delta/3$.\\
Finally, the path $\rrr_\delta$ satisfies $\bigcup_{t \in [t_\delta, T]}\mathcal B_{\kappa_- t_\delta}(\rrr_\delta(t)) \cap \partial \Scal = \emptyset$ and for every  $i\in (1, \dots, J)$  and $a\in \Rr^d\cap \mathrm{span}_{r \in \R}(\gamma^r)$ with $\|a\|<t_\delta \kappa''/2$ there exists $w \in \mathcal W_{i_k,\kappa''}$ such that $\rrr_\delta(\tau_k + \Delta_{k+1}^\beta)+ a = z(\tau_k) + \beta^k t_\delta w $.
\end{lemma}
We defer the proof of this lemma to the end of the section and proceed to present the central estimate allowing us to bound the probability in \eref{e:28} from below --- in the sense of large deviations. To do so, defining throughout $\beta := 3/\kappa''$ and $\xi:= \min(1,(\kappa''/3)^{J+1}/3, \epsilon)$ so that \lref{l:breakup} holds, by triangle inequality it is sufficient to consider the event $\sup_{t \in[0,T]} \|\rrr_\delta(t)- X^v(t)\| < \delta/3$. Furthermore,  $\rrr_\delta \in \mathcal S$ and for any $\dpp\leq \delta' \leq \delta/3$  we can further bound the event of interest from below as follows:
  \begin{align}
    \px{x_0^v}{X^v \in \BBp_{[0,T]}(\delta, \rr)}
    & \geq \px{x_0^v}{X^v \in \BBp_{[0,T]}(\delta/3, \rrr_\delta)}\notag\\
    & \geq\px{x_0^v}{\{X^v \in \BBp_{[0,t_\delta]}(\delta/3, \rrr_\delta)\} \cap \{X^v(t_\delta)\in \BB_{ \dpp}(x_0+t_\delta w_j)\} \cap \{X^v \in \BBp_{[t_\delta, T]}(\delta^\prime, \rrr_\delta)\}} \notag\\
    & \geq \px{x_0^v}{\{X^v \in \BBp_{[0,t_\delta]}(\delta/3, \rrr_\delta)\}\cap \{X^v(t_\delta) \in \BB_{\dpp}(x_0+t_\delta w_j )\} }\label{e:product}\\
    & \qquad \qquad \qquad \qquad \qquad \qquad \qquad \times \inf_{y \in \BB_{\dpp}( x_0+t_\delta w_j)}\p{X^v \in \BBp_{[t_\delta, T]}(\delta^\prime, \rrr_\delta)~|~X^v(t_\delta)=y }\,,\notag
  \end{align}%\lanote{Differences: $\tilde{r}_\delta$ and $\tilde{r}_\delta^y$.}\la{Is it $t_\delta w_j$ or $x+t_\delta w_j$?}\aanote{corected, thanks}
  where in the last inequality we have used the Markov property. In the remainder of the paper we set \begin{equ}
  \delta^\prime  := \kappa_- t_{\delta}/3\qquad \text{and} \qquad\delta'' := t_\delta \kappa'' < \delta'
\end{equ}
   where recalling that $\kappa_- = \min_{j\in \mathcal I^{bd}}\{\kappa_j\} < 1$ and that $\xi < 1$ we must have $\delta'\leq \delta/3$. We note that this choice is compatible with the definition of $\kappa''$ from \aref{a:escape}.
\begin{remark}  \label{r:choiceofdelta'} We pause briefly to motivate our choice of $\delta'$ and $\delta''$: These small parameters are chosen in such a way as to guarantee that the event in the second term in the last line of \eref{e:product} only contains paths that are uniformly bounded away from $\partial \Scal$, as captured by \lref{l:breakup} and depicted in \fref{f:fig2}.
  % Indeed, recalling the definition of $\kappa_j$, we see that  because of \aref{a:escape} a) and our choice of $\dpp\leq \delta'$, if
  % $y\in \BB_{\dpp+\delta'}(x_0+t_\delta w_j)$ then $y\notin \partial \Acal_j$, since $\inf_{x\in\partial \Acal_j}\|x-x_0-t_{\delta} w_j\|\geq \kappa_j t_{\delta}>\delta^{\prime}\geq\dpp$ by definition. Combining this with \lref{l:breakup} we also immediately recognize that
  % $\bigcup_{t \in [t_\delta, T]}\mathcal B_{\delta'}(\rrr_\delta(t)) \cap \partial \Scal = \emptyset$.
\end{remark}
\noindent The desired result is obtained by showing that
\begin{align}
  &\lim_{\delta \to 0}\liminf_{v \to \infty}  \frac 1 v \log \px{x_0^v}{\{X^v \in \BBp_{[0,t_\delta]}(\delta/3, \rrr_\delta)\}\cap \{X^v(t_\delta) \in \BB_{\dpp}(x_0+w_j t_\delta)\} } = 0\qquad\text{and}\label{e:product2}\\
  &\lim_{\delta \to 0}\liminf_{v \to \infty}  \frac 1 v \log \inf_{y \in \BB_{\dpp}( x_0+t_\delta w)}\p{X^v \in \BBp_{[t_\delta, T]}(\delta^\prime, \rrr_\delta)~|~X^v(t_\delta)=y } \geq - I^{x_0}_{\lbrack0,T\rbrack}(z).\label{e:product3}
\end{align}
  %and we recall that $\trr_\delta^y$ is the path of $r$ shifted to $y$ at time $t_\delta$: $\trr_\delta^y(t_\delta) = y$.
 The term in \eref{e:product2} is bounded from below in \sref{s:jumpbounds}, while in \sref{s:rates} and \sref{s:patching} we formulate and combine the estimates in different $\mathcal A_j$ to bound \eref{e:product3}, thereby proving the desired \abbr{LDP} lower bound.

 %\mr{In \sref{s:rates} we show that the difference in costs $|I_{[0,T]}(\rrr_\delta) - I_{[0,T]}(\rr)|$ between the path $\rrr_\delta(t)$ defined above and the original path $\rr$ goes to zero as $\delta \to 0$.}

  %\lanote{Isn't it crucial to have instead $\{\sigma_{\nvn_j}\geq t_\delta\}$? In some sense to be sure that by $t_{\delta}$ the process did enough jumps to be close to $x_0+w_jt_\delta$ but not more than $\nvn_j$ since with $Xi$ we cannot control what happens after the $\nvn_j$-th jump. Am I missing something?}.
  % \la{Am I right in saying that it is not trivial the inclusion above? I would slightly modify the proof of Lemma~\ref{p:start} accordingly. In this sense, we will simply consider the intersection of   $ \{X^v \in \BB_{[0,t_\delta]}(\delta/3, \trr_\delta)\}$ with the event $\Xi_j(\nv,v)$ as follows:}
  % %This implies that
  % \begin{equ}
  % \px{0}{\{X^v \in \BB_{[0,t_\delta]}(\eta, \trr_\delta)\}\cap \Xi_j(\nv,v) },
  %  % \px{0}{X^v \in \BB_{[0,t_\delta]}(\delta/3, \trr_\delta)\cap \{X^v(t_\delta) \in \BB_{\delta K/6}(w_j t_\delta)\} } \geq \px{0}{\Xi_j(\nv,v) \cap \{X^v(t_\delta) \in \BB_{\delta K/6}(w_j t_\delta)\} }
  % \end{equ}
  % for $\eta=\frac{\delta K}{6}$.
  % We bound probabilities such as this one %on the \abbr{rhs}
  % in the next section.

\subsection{Jump bounds} \label{s:jumpbounds}

\begin{figure}[t]
 \centering
 \def\svgwidth{.6\textwidth}
 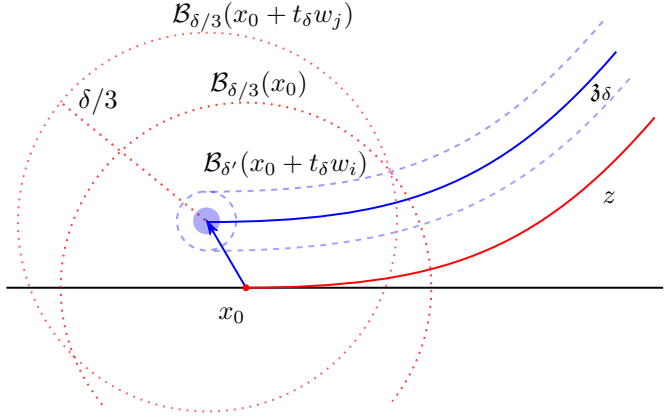
 \caption{Schematic representation of the desired effect for the choice of parameters $\delta, \delta', \delta''$ summarized in \lref{l:breakup}. For a fixed path $\rr$ and $\delta > 0$ by our choice of  $\xi>0$ and consequently $t_\delta$ we have that a neighborhood of the path $x_0+w_{i_0} t$ (blue arrow) is contained in $\bigcap_{t \in (0, t_\delta)} \mathcal B_{\delta/3}(x_0+t w_{i_0})$ (the intersection of the two dotted balls). By our choice of $\delta'(\xi, \kappa_-)>0$,  we find that a $\delta'$-neighborhood (dashed blue region) of the path $\rrr_\delta$ (blue line) on $[t_\delta,T)$ never  intersects $\partial S$ while $\mathcal B_{\delta'}(x_0+w_{i_0}t_\delta) \subset \bigcap_{t \in (0, t_\delta)} \mathcal B_{\delta}(z(t))$. The shaded blue region represents $\mathcal B_{\delta''/2}(x_0+t_\delta w_j)$. }
 \label{f:fig2}
\end{figure}
{We now proceed to bound from below the first term on the last line of \eref{e:product}. To do this we consider a convenient subset of outcomes obtained by fixing a precise sequence of jumps (but not the \emph{times} of the jumps) that the process undergoes in the interval $(0,t_\delta)$.
To define such an event, we recall the definition of the sequence $\Ecal_j$ of $n_j$ jumps leading away from $\partial \Acal_j$ and we denote the event of repeating the sequence of jumps in $\Ecal_j$ $n$ times by
\begin{equ}\label{e:xij}
  \Xi_j(n,v) := \bigcap_{m=0}^{n-1}\bigcap_{i=1}^{n_j} \left\{X^v(\sigma_{m n_j+i}) - X^v(\sigma_{m n_j+i}-) = v^{-1}\gamma^{r_{i}}\right\}\,,
\end{equ}
where, for all $k\in\NN$, $\sigma_k$ is the time of the $k$-th jump of the Markov process $X^v$.
Furthermore, we note that by our choice of $t_\delta < 1/6\delta$  we must have $\{x_0 + t w_j~:~t \in (0,t_\delta)\} \subset \bigcap_{t \in (0, t_\delta)} \mathcal B_{\delta/3}(x_0 + w_j t)$, as depicted in \fref{f:fig2}.
Thus for $\nv_+ := \lfloor \frac{v}{\alpha_j}t_\delta \rfloor$ and $\nv_- := \lceil \frac{v}{\alpha_j}(t_\delta - \delta^\prime)\rceil$ we have for all $v$ large enough that $\nv_- < \nv_+$\,,
\begin{equ}\label{e:subset1}
  \{X^v \in \BBp_{[0,t_\delta]}(\delta/3, \rrr_\delta)\} \supseteq \Xi_j(\nv_+,v)  \cap \{\sigma_{\nv_+ n_j}> t_\delta\}\,,
\end{equ}
and also
\begin{equ}
 \{X^v(t_\delta) \in \BB_{\delta'}(x_0+w_j t_\delta)\}   \supseteq \Xi_j(\nv_+,v) \cap \{\sigma_{\nv_+ n_j}> t_\delta\} \cap\{\sigma_{\nv_- n_j} \leq t_\delta\}\,.
\end{equ}
% \la{What I think is going on above is the following: under the event $\{X^v(t_\delta) \in \BB_{\delta^\prime/3}(x_0+w_j t_\delta)\} \cap \{\sigma_{\nvp n_j}> t_\delta\}$ we have that for $v$  sufficiently large
% \[\Xi_j(\nv,v)\subseteq\{X^v \in \BBp_{[0,t_\delta]}(\delta/3, \rrr_\delta)\},  \]
%  since as $v$ increases, the path of $X^v$ under $\Xi_j(\nv,v)$ in $[0,t_\delta]$ approaches a linear path as a function of $w_j$. But I do not get the passages above. IMPORTANT: I haven't gotten yet why the ball $\BB_{\delta^{\prime}/3}(x_0+w_j t_\delta)$ should be such that $\delta^\prime$ is related to $\eta^{-1}(\delta/3)$ or to $\delta \kappa_j$. See also note above.
% }
Note that, as $v$ and $n$ increase the paths in $\Xi_j(n,v)$ have ranges concentrating on a straight line segment in $\RR^d$
However, there is no information about the speed at which they move along this line segment.
This degree of freedom will be sufficient to establish the \abbr{LDP} lower bound, as we shall see next.
In preparation for the next result observe by \aref{a:rate-converg}a) that $\limsup_{v \to \infty}\sup\left\{ \sum_{r\in \mathcal R} \Lambda_r^v(x) \colon x \in \Bcal_{2 t_\delta}(x_0)\right\} < \infty$ and define $\bar t_\delta(\alpha, \epsilon''):=\min_j\{\alpha_j/n_j, \epsilon^{\prime\prime} / \max_{r\in\Rcal}\Lip(\lambda_r)\}$.

\begin{lemma}\label{p:start}
Suppose $x_0 \in  \Acal_j$ and let \aref{a:rate-converg} and~\ref{a:escape} hold and that $\delta$ is small enough that $t_\delta < \bar t_\delta(\alpha,\epsilon'')$, then for
$\bar\lambda > \max\{ 1, \limsup_{v \to \infty} \sup\left\{ \sum_{r \in \mathcal R} \Lambda_r^v(x) \colon x \in \Bcal_{2 t_\delta}(x_0)\right\}\}$
we have %Suppose further that $\int_0^{\epsilon}\abs{\log\left(\lambda_{r^\ast}\left(x +s \gamma^{r^\ast}\right)\right)} \d s < \infty$, then, if $(\delta - \eta, \delta + \eta) \subset (x_i,\epsilon)$
\begin{multline*}
\liminf_{v \to \infty} \frac1v \log\PP\left[X^v(t_\delta) \in \BB_{\dpp}(x_0+t_\delta w_j)\, ,\, \Xi_j(\nv,v)\,,\, \sigma_{\nvp n_j}> t_\delta \middle| X^v(0) = x_0^v \right] \\
\geq
-t_\delta \left(\frac{n_j}{\alpha_j}\log\left(\frac{n_j}{\alpha_j \overline{\lambda}}\right) - \frac{n_j}{\alpha_j} + \overline{\lambda}\right)
+\sum_{i = 1}^{n_j}\int_0^{t_\delta/\alpha_j}\log\left(\lambda_{r_{i}}\left(x_0 +s \alpha_j w_j\right)\right) \d s
+t_\delta \frac{n_j}{\alpha_j} \log\aleph \,.
\end{multline*}

\end{lemma}

To prove the above result, defining throughout $\tilde \gamma^{(i)} : = \sum_{k = 1}^{i-1} \gamma^{r_k}$ we introduce the following lemma relating the Riemann sum of $\log \lambda_r$ along the escape sequence defining $\Xi_j$ to the corresponding integral.
\begin{lemma}\label{l:intconv}
Suppose $x_0 \in  \Acal_j$, let \aref{a:rate-converg} and~\ref{a:escape} hold, then for all $\delta$ such that $t_\delta < \bar t_\delta(\alpha,\epsilon'')$ and $r_i \in \mathcal E_j$ we have
  \begin{equ}\label{e:moreprecise}
    \liminf_{v \to \infty} \frac1v \sum_{m=0}^{\nv_+} \log \left({\lambda_{r_{i}}\left(x_0^v+ \frac m v  \alpha_j w_j + v^{-1}\tilde \gamma^{(i)}\right)}\right) \geq \int_0^{t_\delta/\alpha_j} \log\left(\lambda_{r_{i}}\left(x_0 +s \alpha_j w_j\right)\right) \d s\,.
  \end{equ}
\end{lemma}

\begin{proof}[Proof of \lref{p:start}]
   When $x_0 \not \in \partial \mathcal A_j$ all the rates are strictly positive by definition, so the result follows by standard large-deviation estimates \cite{ShwartzWeiss95}. Therefore, for the rest of the proof we assume that $x_0 \in \partial \Acal_j$.
 Denoting by $\II\left\{x \in A\right\}$ the indicator function on the set $A$,  we introduce a jump $r^\star$ with
  \begin{equ}
    \Lambda_{r^\star}^{v}(x) := \left(\bar \lambda - \sum_{r \in \R} \Lambda_r^{v}(x)\right)\II\left\{x \in \Bcal_{2 t_\delta}(x_0)\right\} \qquad \text{and} \qquad \gamma^{r^\star} = 0\,,
    \end{equ}
    and expand the set of jumps $\R^\star := \R \cup \{r^\star\}$.
We then define a new family of processes $\bar X^v$ on the extended set of jumps $\R^\star$ and corresponding jump rates.
By independence of the jump processes we trivially couple the underlying Poisson processes for jumps in $\R$ to the ones of the process $X^v$, so that $\bar X^v(t) =  X^v(t)$ a.s. and proceed to establish the desired result for $\bar X^v$. In the rest of the proof, by abuse of notation we will denote by $\bar \Xi_j$ the set defined in \eref{e:xij} for $\bar X^v$ instead of  $X^v $.
We then have
\begin{multline*}
\PP\left[\bar X^v(t_\delta) \in \BB_{\dpp}(x_0+t_\delta w_j), \sigma_{\nvp n_j}> t_\delta, \bar \Xi_j(\nv_+,v) \middle| \bar X^v(0) = x_0^v\right]  \\ \geq
\PP\left[\sigma_{\nv_- n_j} \leq t_\delta < \sigma_{\nv_+ n_j} \middle| \bar \Xi_j(\nv_+,v) , \bar X^v(0) = x_0^v\right]
\times \PP\left[ \bar \Xi_j(\nv_+,v)  \middle| \bar X^v(0) = x_0^v\right]\,.
\end{multline*}

On the event $\bar \Xi_j(\nv_+,v) \cap \left\{\sigma_{\nv_+ n_j} \geq t \right\}$ one has, for $v$ large enough,
\begin{equation*}
\sup_{s < t}\norm{\bar X^v(s) - x_0} \leq \norm{x_0 - x_0^v} + t_\delta \norm{w_j} + v^{-1}\max_i\norm{\tilde \gamma^{(i)}}
< 2 t_\delta\,,
\end{equation*}
and therefore $\sum_{r \in \Rcal^\ast} \Lambda_r^v(\bar X^v(s)) = \bar \lambda$ for all $s < t$.

Now, recalling that $\alpha_j w_j := \sum_{k = 1}^{n_j} \gamma^{r_{k}}$ with $\|w_j\| = 1$, by the conditional independence of the jumps
\begin{equation}
\PP\left[ \bar \Xi_j(\nv_+,v) \middle|\bar X^v(0) = x_0^v\right] =
\prod_{m=0}^{\nv_+ - 1} \prod_{i=1}^{n_j} \frac{\Lambda^{v}_{r_{i}}\left(x_0^v+ m v^{-1}  \alpha_j w_j + v^{-1}\tilde \gamma^{(i)}\right)}{\overline{\lambda}}\,.
\end{equation}
and thus
\begin{multline}
\frac1v \log \PP\left[ \bar \Xi_j(\nv_+,v) \middle| \bar X^v(0) = x_0^v\right]
\geq
\sum_{i = 1}^{n_j} \frac1v \sum_{m=0}^{\nv_+ - 1}
  \log \left(\frac{\Lambda^{v}_{r_{i}}\left(x_0^v+ m v^{-1}  \beta_j w_j + v^{-1}\tilde \gamma^{(i)}\right)}
                       {\lambda_{r_{i}}\left(x_0^v+ m v^{-1}  \alpha_j w_j + v^{-1}\tilde \gamma^{(i)}\right)}\right) \\
 % + \sum_{i = 1}^{n_j} \frac1v \sum_{m=1}^{\nv}
 %   \log \left(\frac{\lambda_{r_{i}}\left(x^v+ m v^{-1}  \alpha_j w_j + v^{-1}\sum_{k = 1}^{i-1} \gamma^{r_{k}}\right)}
 %                        {\lambda_{r_{i}}\left(x^v+ m v^{-1}  \alpha_j w_j \right)}\right)
 + \sum_{i = 1}^{n_j} \frac1v \sum_{m=0}^{\nv_+ - 1}
   \log \left(\frac{\lambda_{r_{i}}\left(x_0^v+ m v^{-1}  \alpha_j w_j + v^{-1}\tilde \gamma^{(i)}\right)}{\overline \lambda}\right)\,. \label{e:2}
\end{multline}
% Using \aref{a:rate-converg} to bound the rates from below and from above, along with the continuity of $\lambda_r$, it follows by Fatou that

Now, using \lref{l:intconv} and recalling the definition of $\aleph>0$ from \aref{a:rate-converg} we have for \eref{e:2}
\begin{equ}\label{e:jumptype_ldp}
\liminf_{v \to \infty} \frac1v \log \PP\left[ \bar \Xi_j(\nv,v) \middle| \bar X^v(0) = x_0^v\right]
=
\sum_{i = 1}^{n_j}\int_0^{t_\delta /\alpha_j}\log\left(\lambda_{r_{\ell_i}}\left(x_0 +s \alpha_j w_j\right)\right) \d s
  -n_j t_\delta \log\pc{\overline{\lambda}/ \aleph} / \alpha_j.
\end{equ}
Also since the waiting times between jumps are independent of which type of jump actually occurs
\begin{equ}
\PP\left[\bar X^v(t_\delta ) \in \BB_{\dpp}(x_0+t_\delta w_j), \sigma_{\nvp n_j}> t_\delta\middle| \bar \Xi_j(\nv,v) , \bar X^v(0) = x_0^v\right]
=\PP\left[\nv_- n_j  \leq Y^v < \nv_+ n_j\right],
\end{equ}
where $Y^v$ is Poisson distributed with mean $t_\delta v\overline{\lambda}$.
By definition, $\bar \lambda\geq 1$ so that it is in particular greater than $\nv_+ n_j / v$ for our choice of $t_\delta$, and we have
\begin{equation}\label{e:jumpcount_ldp}
\lim_{v \to \infty} \frac1v \log \PP\left[\nv_- n_j  \leq Y^v < \nv_+ n_j\right]
=
-t_\delta \left(\frac{n_j}{\alpha_j} \log\left(\frac{n_j}{\alpha_j \bar \lambda}\right) - \frac{n_j}{\alpha_j} + \bar\lambda\right)
\end{equation}

The result now follows by combining \eqref{e:jumpcount_ldp} and \eqref{e:jumptype_ldp}.
\end{proof}

\begin{proof}[Proof of \lref{l:intconv}]
We  note that $\lambda_r$ is Lipschitz so whenever $\lambda_r(x_0)>0$ then  $\lambda_r(x+ s w_j)$ is uniformly bounded away from 0 on a sufficiently small time interval and the result follows immediately by a dominated convergence argument.
We therefore assume throughout that $\lambda_r(x_0) = 0$.
 Introducing
\begin{equation}\label{e:tnv}
\tnv := \inf \pg{m \in \mathbb N~:~\frac{x_0^v+ m v^{-1}  \alpha_j w_j + v^{-1}\tilde \gamma^{(i)} - x_0}{ \|x_0^v+ m v^{-1}  \alpha_j w_j + v^{-1}\tilde \gamma^{(i)} - x_0\|}\in \mathcal W_{j,\kappa''} }\,,
\end{equation}
we split the sum in the statement of the Lemma into the terms $m = 0$, $m\in(0,\dots, \tnv)$ and $m \in (\tnv+1, \dots,\nv_+)$ and proceed to bound their contribution separately.
The term $m=0$ is automatically bounded by \aref{a:escape}b). For the second term we observe since  $\log \left({\lambda_{r_{i}}\left(x_0^v+ m v^{-1}  \alpha_j w_j + v^{-1}\tilde \gamma^{(i)} \right)}\right) $ is increasing in $m$ by Assumption~\ref{a:escape}d)  that
\begin{multline*}
\liminf_{v \to \infty} \frac1v \sum_{m=1}^{ \tnv} \log \left({\lambda_{r_{i}}\left(x_0^v+ m v^{-1}  \alpha_j w_j + v^{-1}\tilde \gamma^{(i)} \right)}\right) \\
\geq
\liminf_{v \to \infty} \frac1v \int_0^{\tnv / v}  \log \left({\lambda_{r_{i}}\left(x_0^v + v^{-1}\tilde \gamma^{(i)} + t \alpha_j w_j\right)}\right) \d t
= 0\,,
\end{multline*}
where the final equality arises since $\lim_{v \to \infty} x_0^v = x_0$ implies $\lim_{v \to \infty} v^{-1} \tnv = 0$, and we have the integral estimate from Assumption~\ref{a:escape}c), which is uniform in the starting points $x_0^v + v^{-1}\tilde  \gamma^{(i)} \in Z_0^v$.

%  Noting that since $\lim_{v \to \infty} x_0^v = x_0$ we must have $\lim_{v \to \infty} v^{-1} \tnv = 0$ and combining the assumed uniform integrability and increasing property of $\lambda_r(x_0^v + s w_j + v^{-1}\tilde \gamma^{(i)})$ in $s$ from \aref{a:escape} we have

 Similarly for the terms $m \in (\tnv+1, \dots,\nv_+)$, defining $m' := m - \tnv \geq 0$ we can write
 \begin{equation}\label{e:increment}
 x_0^v+ m v^{-1}  \alpha_j w_j + v^{-1}\tilde \gamma^{(i)}
 =
 x_0 + m'v^{-1}  \alpha_j w_j +  \pc{x_0^v+ \tnv v^{-1}  \alpha_j w_j + v^{-1}\tilde \gamma^{(i)} - x_0}\,,
 \end{equation}
and note that by \eqref{e:tnv} the vector in brackets, once renormalized, is in $\mathcal W_{j,\kappa''}$.\\
We also have $\lim_{v\to \infty}\norm{x_0^v+ \tnv v^{-1}  \alpha_j w_j + v^{-1}\tilde \gamma^{(i)} - x_0} = 0$ so provided that $\delta$ is small enough that $x_0 +m^\prime v^{-1} \alpha_j w_j \in \Acal_j$
%\aanote{I have assumed that the interior of $\mathcal A_i$ covers $\mathcal S$, I think this is sufficient.}
and
%\footnote{Take $t_\delta < \epsilon^{\prime\prime}/ \Lip(\lambda_{r_i})$ since $m v^{-1} \alpha_j \leq t_\delta$.}
$\lambda_{r_i}(x_0 +m^\prime v^{-1} \alpha_j w_j ) < \epsilon^{\prime\prime}$ we may apply Assumption~\ref{a:escape}d) to see that  for each $m'\geq 0$
\begin{equ}
{\lambda_{r_{i}}\left(x_0^v+ m v^{-1}  \alpha_j w_j + v^{-1}\tilde \gamma^{(i)} \right)} \geq
{\lambda_{r_{i}}\left(x_0+ m' v^{-1}  \alpha_j w_j \right)} .
\end{equ}
A second application of Assumption~\ref{a:escape}d) implies
\begin{equ}
\liminf_{v \to \infty} \frac1v \sum_{m'=1}^{\nv_+- \tnv} \log \left({\lambda_{r_{i}}\left(x_0+ m v^{-1}  \alpha_j w_j \right)}\right)
\geq
\liminf_{v \to \infty}  \int_{0}^{(\nv_+-\tnv) / v}  \log \left({\lambda_{r_{i}}\left(x_0 + t \alpha_j w_j\right)}\right) \d t.
\end{equ}
%
%  We start by writing the \abbr{lhs} of \eref{e:moreprecise} above and noticing that for all $i=1,\dots,n_j$ we have, by \aref{a:escape} a) that it equals
%  \begin{equ}\label{e:3}
%    0 + \liminf_{v \to \infty} \frac1v \pg{\sum_{m=1}^{ \tnv} \log \left({\lambda_{r_{i}}\left(x_0^v+ m v^{-1}  \alpha_j w_j + v^{-1}\tilde \gamma^{(i)} \right)}\right) + \sum_{m=  \tnv}^{\nv} \log \left({\lambda_{r_{i}}\left(x_0^v+ m v^{-1}  \alpha_j w_j + v^{-1}\tilde \gamma^{(i)}\right)}\right)}
%  \end{equ}
%     At the same time, we can write the second term in curly brackets in \eref{e:3} as
%  \begin{equ}
%    \liminf_{v \to \infty} \frac1v \sum_{m=  \tnv}^{\nv} \log \frac{\lambda_{r_{i}}\left(x_0^v+ \frac m v  \alpha_j w_j + v^{-1}\tilde \gamma^{(i)}\right)}{\lambda_{r_{i}}\left(x_0+ \frac {m-\tnv} v  \alpha_j w_j\right)} + \log {\lambda_{r_{i}}\left(x_0+ \frac {m-\tnv} v  \alpha_j w_j\right)} \geq \liminf_{v \to \infty} \frac1v \sum_{m=  0}^{\nv-\tnv}  \log {\lambda_{r_{i}}\left(x_0+ \frac m v  \alpha_j w_j\right)}\,,
%  \end{equ}
%  where recalling that $\lambda_{r_{i}}\left(x_0^v+ \frac m v  \alpha_j w_j + v^{-1}\tilde \gamma^{(i)}\right) \geq \lambda_{r_{i}}\left(x_0+ \frac {m-\tnv} v  \alpha_j w_j\right)$ by \aref{a:escape} d) \blue{RP: I do not see how \aref{a:escape} b) helps here.} and our choice of $\tnv$ we have used Fatou's lemma to bound the first term on the \abbr{lhs}.
%  Finally, the convergence of the \abbr{rhs} of the above bound given the assumed integrability of $\log \lambda_r$ from \aref{a:escape} c) gives the desired result.
\end{proof}

\subsection{Approximation of the rate functional} \label{s:rates}
% {As explained in \cite{dupuis91}, upon assuming \aref{a:exptight}, the difficulties in proving the desired result are purely technical. Since the proof relies on a small shift of the original path in order to bound its distance uniformly away from the boundary, we have on one hand to bypass the fact that the shifted path may exit the domain of interest $\SS^*$, resulting in an infinite rate function. On the other hand, the approximating functional $\ell^\delta$ used traditionally \cite{ShwartzWeiss95} behaves poorly under such shifts. To circumvent this issue, we rely on alterntative approximating functionals $\ell^\Delta$, which we define below, exactly as in \cite{ShwartzWeiss05}.}
%
% {Except for the necessary modifications listed above, the proof of the upper bound is mainly standard:  we first approximate a path $r$ in a set $K$ with a series of increasingly fine piecewise linear interpolation. Then, for every interval of such interpolation we estimate the probability of interest by approximating the original process with a constant coefficients one. As anticipated above, we closely follow the steps of \cite[Section 4]{ShwartzWeiss05} in realizing this program.}
After the process has left the boundary $\partial \Scal$, we estimate the cost of a given path $z$ by approximating $z$ with another path, that is uniformly bounded away from $\partial \Scal$. This implies that a standard \abbr{LDP} holds for such shifted path, and this can then be used to bound the rate function of the original $z$.
We start by proving an adaptation of \cite[Lemma 4.1]{ShwartzWeiss05} to the present setting, recalling that $\omega_\rr$ is the modulus of continuity of $\rr$.
\begin{lemma} \label{l:1}
  Under Assumption \ref{a:escape}, for every $x \in  \Acal_j$ for $j \in \I$ recalling that $t_\delta :=\frac 1 6 \xi \min(\omega_\rr^{-1}(\delta), \delta)$ the cost of the path $\rr_\delta(t) = x + t w_j$ satisfies
\begin{equ}
  \lim_{\delta \to 0} I_{[0,t_\delta]}(\rr_\delta) = 0\,.
\end{equ}
% Furthermore, whenever $\rr$ is uniformly bounded away from $\{\partial \Acal_j\}$, there exists $C> 0$ such that
% \begin{equ}
%   I_{[0,t_\delta]}(\rr) \leq  C \delta
% \end{equ}
\end{lemma}

{
\begin{proof}
  The thesis follows immediately by \aref{a:escape}, resulting in the integrability of the rate functional $I_{\lbrack0,t_\delta\rbrack}$ along the chosen trajectories.
  % The proof of the second part of the lemma is identical to the one of \cite[Lemma 4.1]{ShwartzWeiss05}, where the authors assume that the rates of relevant reactions are uniformly bounded away from $0$, as it is the case when $d(\rr,\partial \Acal_j)>\kappa_j>0$.
\end{proof}
% We now note that the only part of the proof of \cite[Theroem 4.1]{ShwartzWeiss05} that uses the boundedness from below of the rates \cite[Assumption 2.2 C]{ShwartzWeiss05} is the one leading to \cite[Corollary 4.2]{ShwartzWeiss05}. Indeed, all lemmas used subsequently  are from \cite{ShwartzWeiss95}. \aanote{check lemmas one by one: still to do} Hence, it only remains to prove that \cite[Lemmas 4.2, 4.3]{ShwartzWeiss05}, as \cite[Lemmas 4.4]{ShwartzWeiss05} only assumes continuity and boundedness of the rates. To state the lemmas, we need to introduce the following alternative cont functionals:
% \begin{equs}
%   I_{[0,T]}^\Delta(\rr) &= \int_0^T \ell^\Delta(r_s,\dot r_s) \, \d s
%   \qquad \qquad I_{[0,T]}^m(\rr) &= \int_0^T \ell^m(r_s,\dot r_s) \, \d s
% \end{equs}
% where, defining throughout  $\lambda^m(x):=\lambda_r(x) \vee m^{-1}$ for all $r \in \R$,
% \begin{equs}
%   % \ell^\Delta(x,y) &= \sup_{\theta \in \Rr^d} \pq{\theta \cdot y - \sum_{r \in \R} \sup_{z_r \in \BB_\delta(x)} \pq{\lambda_r(z_r) \pc{e^{\theta\cdot \gamma_r}-1}} }\\
% \ell^m (x,y)& = \sup_{\theta \in \Rr^d} \pq{\theta \cdot y - \sum_{r \in \R} {\lambda_r^m(x) \pc{e^{\theta\cdot \gamma_r}-1}} }
% \end{equs}
We define throughout
\begin{equation}\label{e:ell}
\ell(x,y):=\sup_{\theta\in \RR^{d}}{\theta \cdot y - \sum_{r \in \R} \lambda^r(x) \pc{\exp\pc{\theta \cdot \gamma_r}-1}}.
\end{equation}
and recall that, by convex duality, for any $x,y \in \Rr^{d}$ we have ${\ell(x,y) = \inf_{\{\mu \in \RR_{\geq0}^{|\mathcal R|}~:~\sum_{r\in \mathcal R} \mu_r \gamma^r = y\}} \Hcal(\mu|\lambda(x))}$ for  $ \Hcal(\mu|\lambda)$ defined in \eref{e:lbrf2}, as proven \eg in \cite{ShwartzWeiss95}. Consequently we can express $I_{[0,T]}(z) = \int_0^T \ell(z(s), z'(s))\,\d s$. This allows to prove the following adaptation of \cite[Lemma 5.1]{ShwartzWeiss05}.
\begin{lemma}\label{l:42}
  Let Assumptions~\ref{a:rate-converg} and \ref{a:escape} hold. Fix $i \in \mathcal I$, $\tau > 0$ and let the path $\rr$ take values in $\Acal_i$ for $t \in [0,\tau]$ and satisfy $I_{[0,\tau]}(\rr)<K$ for $K<\infty$. Then for any $w \in \mathcal W_{i,\kappa''}$, $C_\beta > 0$ the shifted path $\rr_\delta(\cdot) = \rr(\cdot)+ C_\beta t_\delta w$ satisfies
\begin{equ}
  \limsup_{\delta \to 0} I_{[0,\tau]}(\rr_\delta) \leq I_{[0,\tau]}(\rr)
\end{equ}
\end{lemma}
\begin{proof}
  We define $\ell_1(t) = \ell(\rr(t),\rr'(t))$ and $\ell_2(t) = \ell(\rr_{\delta}(t),\rr_{\delta}'(t))$ and denote by $(\mu_r^*(t))_{r \in \R}$  the optimizing set of jumps in \eref{e:lbrf} for the path $\rr$.
This minimizer exists, because the sublevel sets for $\mu \mapsto \Hcal(\mu | \lambda)$ for $\mathcal H$ from \eref{e:lbrf2} are compact.
  Then  we have that $\ell_2(t)\leq \mathcal H(\mu^*(t)|\lambda(\rr_\delta(t)))$.
  On the other hand, by continuity of the asymptotic rates $\lambda_r$ there exists a function $K_\lambda(\delta)$ with $\lim_{\delta \to 0} K_\lambda(\delta) = 0$ for which we have $|\lambda_r(\rr(t)) - \lambda_r(\rr_\delta(t))| < K_\lambda(\delta)$ for all $r \in \R$, so that
  \begin{equ}
    \ell_2(t) - \ell_1(t) \leq \sum_{r \in \R}   \lambda_r(\rr_\delta(t)) - \lambda_r(\rr(t)) + \mu_r^*(t) \log \frac{\lambda_r(\rr(t))}{ \lambda_r(\rr_\delta(t))} \leq |\R|K_\lambda(\delta) + \sum_{r\in \mathcal R}\mu_r^*(t) \log \frac{\lambda_r(\rr(t))}{ \lambda_r(\rr_\delta(t))}\,.
  \end{equ}
  We now bound the second term on the \abbr{RHS} from above depending on whether $\lambda_r(\rr(t)) > \epsilon''$ from \aref{a:escape}. If $\lambda_r(\rr(t)) \leq \epsilon''$, by the assumed increasing property of $\lambda_r(\cdot)$ along $x+s w $, we have $\log {\lambda_r(\rr(t))}/{ \lambda_r(\rr_\delta(t))} \leq 0$. On the other hand, if $\lambda_r(\rr(t)) > \epsilon''$ then for $\delta$ small enough
  \begin{equ}
    \log \frac{\lambda_r(\rr(t))}{\lambda_r(\rr_\delta(t))} \leq \log \frac {\lambda_r(\rr(t))}{\lambda_r(\rr(t)) - K_\lambda(\delta)}\leq \log \frac {\epsilon''}{ \epsilon'' - K_\lambda(\delta)} \leq \frac {2 K_\lambda(\delta)}{\epsilon''}\,,
  \end{equ}
  where in the last inequality we used $\log(1-x)^{-1} < 2x$ for $x$ small enough. It remains to show that the contribution of the term $\mu_r^*(t)$ is bounded from above on the paths of interest. This result is obtained in the proof of \cite[Lemma 5.1]{ShwartzWeiss05} by the convexity and asymptotic growth of the Lagrangian $\ell(x, y)$ in its second argument, proven in \cite[Lemma 5.1]{ShwartzWeiss05}, \cite[Lemma 5.17]{ShwartzWeiss95} leveraging only the boundedness of the rates $\lambda_r$. Following the same argument we bound $\mu_r^*(t) \leq C_0(1+\ell_1(t))$ and we finally obtain
  \begin{equ}
      \ell_2(t) \leq  \ell_1(t) + K_\lambda(\delta)(C_1 + C_2 \ell_1(t))
  \end{equ}
  for sufficiently large, positive constants $C_0, C_1, C_2$. By the assumed boundedness of $I_{[0,\tau]}(\rr)$, this gives the desired result by integration.
\end{proof}

We now combine the above estimates, established in each region $\mathcal A_i$ separately, to obtain convergence of the rate functional $I_{[0,T]}(\rrr_\delta)$ to $I_{[0,T]}(z)$ as $\delta \to 0$. While the idea of the proof is the same as in the original reference, we have to reproduce the process more closely, as in our case we cannot bound, in general, the rate function of the shifts linearly in $\delta$ as done in \cite[Lemma 4.1]{ShwartzWeiss05} and we only have the limiting result \lref{l:1}. To bypass this issue, we leverage exponential tightness -- discussed directly below the statement of \tref{th:main} -- to show that the number of transitions between different $\Acal_j$ done by the path of interest is bounded uniformly on sublevel sets of the rate functional. For any $a \in \mathbb R^d$, we extend the definition of $\rrr_\delta$ on the interval $[0,T]$ as:
\begin{equ}
  \tilde \rrr_\delta^a(t) = \begin{cases}\rrr_\delta(t_\delta) +a \qquad & \text{ for } t\in[0,t_\delta)\\
\rrr_\delta(t)+a & \text{ for } t\in[t_\delta,T]\end{cases}\,.
\end{equ}

\begin{lemma}\label{l:43}
  Let Assumptions~\ref{a:rate-converg} and \ref{a:escape} hold, and let the path $\rr$ satisfy $I_{[0,T]}(\rr)<\infty$. Then  the path $\tilde \rrr_\delta$  satisfies
\begin{equ}
  \limsup_{\delta \to 0}\,\,\sup_{a \in \mathrm{span}_{r\in \R}(\gamma^r) ~:~\|a\|<\kappa''t_\delta} I_{[0,T]}(\tilde \rrr_\delta^a) \leq I_{[0,T]}(\rr)
\end{equ}
\end{lemma}

\begin{proof}
Recall the definition of times $\tau_1\, \dots, \tau_J$ from \lref{l:breakup}, separating $[0,T]$ in a finite number of intervals where the path $\rr$ is contained in a set $\Acal_j$.
%We show that the cost of the shifted path satisfies $I_{[0,T]}(\rrr_\delta)< K$, assuming that $\eta$ was chosen small enough for the above to hold.
We now express the rate function as the sum of the cost of the shifted path and the cost of the shifts: For $\Delta_i := t_\delta\sum_{k=0}^i (3/\kappa'')^k$ (reflecting the choice of $\beta$ in \lref{l:breakup}) we write
     \begin{equ}\label{e:split}
       I_{[0,T]}(\tilde \rrr_\delta^a) = I_{[0,t_\delta)}(\tilde \rrr_\delta^a) + \sum_{i=1}^J I_{[\tau_{i-1} + \Delta_i,\tau_{i}+ \Delta_i]}(\tilde \rrr_\delta^a)  + \sum_{i=1}^J I_{[\tau_{i} + \Delta_i,\tau_{i}+ \Delta_{i+1}]}(\tilde \rrr_\delta^a)\,,
       \end{equ}
       and proceed to bound the terms on the \abbr{rhs} separately. For the first term we can trivially choose the optimizing set of fluxes in \eref{e:lbrf2} as $\mu^* = 0$\rp{}, so that by boundedness of the rates $\lambda_r(x)<C$ we have $I_{[0,t_\delta)}(\tilde \rrr_\delta^a) \leq |\R| C t_\delta$, which vanishes with $\delta \to 0$.

       We proceed to bound the summands in the second term. Recalling by \lref{l:breakup} that for each time interval $[\tau_{i-1} + \Delta_i,\tau_{i}+ \Delta_i]$ the trajectory of  $\tilde \rrr_\delta^a$ corresponds to the one of $z + wt_\delta$ for $w \in \mathcal W_{j_i, \kappa''}$ we see that for each such interval we can apply \lref{l:42}. Combining this result with the time-translation invariance of the rate functional we obtain that
      % , for $m$ large enough, $I_{[\tau_{i-1} + \delta_i,\tau_{i}+ \delta_i]}(\rrr_\delta^a) < I^m_{[\tau_{i-1} ,\tau_{i}]}(\rr) + \tilde \epsilon(\eta)$. Therefore,
    %upon choosing $\delta$ small enough we can bound the first term of \eref{e:split} as
    \begin{equ}\label{e:shifted}
       \limsup_{\delta \to 0}\sum_{i=1}^J I_{[\tau_{i-1} + \Delta_i,\tau_{i}+ \Delta_i]}(\tilde\rrr_\delta^a) \leq \sum_{i=1}^J I_{[\tau_{i-1} ,\tau_{i}]}(\rr)\,.
     \end{equ}

     We then bound the third term of \eref{e:split} by \lref{l:1}, recalling that by \lref{l:breakup} the path $\tilde \rrr_\delta^a$ is in $\mathcal S$.
     %To do so we separate two cases, depending on whether $\inf_{x \in \partial \Acal_{j_{k+1}}}\|\rr(\tau_k)-x\|>\epsilon_{j_{k+1}}'$. If this is the case, since the the jump rates are uniformly bounded away from zero in the $\epsilon_j'$-interior of $\Acal_j$, \aanote{here mention that we are in a compact} by \lref{l:1} there exists a universal $C$ such that $I_{[\tau_{i} + \Delta_i,\tau_{i}+ \Delta_{i+1}]}(\rrr_\delta^a) < C(\Delta_{i+1}- \Delta_i)$. On the other hand, if  $\inf_{x \in \partial \Acal_{j_{k+1}}}\|\rr(\tau_k)-x\|<\epsilon_{j_{k+1}}'$
     We start by writing
     \begin{equ}
       I_{[\tau_{i} + \Delta_i,\tau_{i}+ \Delta_{i+1}]}(\tilde\rrr_\delta^a) \leq  \sum_{r \in \mathcal R} \int_{\tau_{i} + \Delta_i}^{\tau_{i}+ \Delta_{i+1}}\!\Big\lbrack  \lambda_r(\tilde\rrr_\delta^a(s)) - \mu_r^* + \mu_r^* \log \frac {\mu_r^*}{\lambda_r(\tilde\rrr_\delta^a(s))}\Big\rbrack\,\d s\,.
     \end{equ}
     where $\mu_r^*$ is given by the multiplicity of reaction $r$ in $\mathcal E_j$. We then further divide the sum on $\R$ based on whether the jump rates $\lambda_r(z(t))$ are bounded from below by $\epsilon''$ from \aref{a:escape} on the time interval of interest.
      %and those with rates going to $0$ when approaching $\partial \Scal$.
      We denote the jumps whose rates do not satisfy this lower bound by $\R^{(0)}(j_i)$ and write
     \begin{equs}
       I_{[\tau_{i} + \Delta_i,\tau_{i}+ \Delta_{i+1}]}(\tilde\rrr_\delta^a) &\leq  \sum_{r \in \R^{(0)}(j_i)} \int_{\tau_{i} + \Delta_i}^{\tau_{i}+ \Delta_{i+1}}\!\Big\lbrack \lambda_r(\tilde\rrr_\delta^a(s)) - \mu_r^* + \mu_r^* \log \frac {\mu_r^*}{\lambda_r(\tilde\rrr_\delta^a(s))}\Big\rbrack\, \d s \\&\qquad \qquad \qquad + \sum_{r \in \R\setminus \R^{(0)}(j_i)} \int_{\tau_{i} + \Delta_i}^{\tau_{i}+ \Delta_{i+1}} \!\Big\lbrack \lambda_r(\tilde\rrr_\delta^a(s)) - \mu_r^* + \mu_r^* \log \frac {\mu_r^*}{\lambda_r(\tilde\rrr_\delta^a(s))}\Big\rbrack\,\d s\,.\label{e:lasteq}
     \end{equs}
     Then, by compactness given by $I(\rr)<K$ there exists $C'(K)>0$ such that the second term is bounded from above by
     \begin{equ}
       \sum_{r  \in \R\setminus \R^{(0)}(j_i)} \int_{\tau_{i} + \Delta_i}^{\tau_{i}+ \Delta_{i+1}} \!\Big\lbrack \lambda_r(\tilde\rrr_\delta^a(s)) - \mu_r^* + \mu_r^* \log \frac {\mu_r^*}{\lambda_r(\tilde\rrr_\delta^a(s))} \Big\rbrack\,\d s < |\R| C' (\Delta_{i+1}- \Delta_i)\,.
     \end{equ}
     On the other hand, for the first term in \eref{e:lasteq} we have
     \begin{equ}
       \sum_{r \in \R^{(0)}(j_i)} \int_{\tau_{i} + \Delta_i}^{\tau_{i}+ \Delta_{i+1}} \!\Big\lbrack \lambda_r(\tilde\rrr_\delta^a(s)) - \mu_r^* + \mu_r^* \log \frac {\mu_r^*}{\lambda_r(\tilde\rrr_\delta^a(s))}\Big\rbrack\,\d s< |\R| C' (\Delta_{i+1}- \Delta_i) + |\R| C''\int_{\tau_{i} + \Delta_i}^{\tau_{i}+ \Delta_{i+1}} \log \lambda_r(\tilde\rrr_\delta^a(s))\,\d s\,,
     \end{equ}
     and using \aref{a:escape} b) and c) we have for some $x \in \Acal_{j_i}$
     \begin{equ}
     \lim_{\delta \to 0} \int_{\tau_{i} + \Delta_i}^{\tau_{i}+ \Delta_{i+1}} \log \lambda_r(\tilde\rrr_\delta^a(s))\,\d s \leq  \lim_{\delta \to 0} \int_0^{\beta^it_\delta} \log \lambda_r(x + s w_{j_i}) \, \d s= 0\,.
   \end{equ}
   Combining the above upper bounds for each transition we have
   \begin{equ}\label{e:shifts}
     \lim_{\delta \to 0} \sum_{i=0}^J I_{[\tau_{i} + \Delta_i,\tau_{i}+ \Delta_{i+1}]}(\tilde\rrr_\delta^a) \leq J \lim_{\delta \to 0} \sup_{i \in (0,\dots, J)} I_{[\tau_{i} + \Delta_i,\tau_{i}+ \Delta_{i+1}]}(\tilde\rrr_\delta^a) = 0\,.
   \end{equ}
   Finally, combining \eref{e:split} with \eref{e:shifted} and \eref{e:shifts} we obtain the desired result.\end{proof}

\subsection{Proof of LDP in path space} \label{s:patching}

% {The above rate function goes to $0$ as expected when $T \to 0$ but it becomes negative for a certain interval of values of $T$, so I may have made a mistake in the comptations..}
% \mr{Indeed there was a small error, I fixed it and now it's non-negative. It's a very interesting function by the way: increasing, concave for $T\leq 1$ and convex for $T\geq1$.}

\begin{proof}[Proof of \pref{l:28}] We conclude the proof by bounding the terms in  \eref{e:product}. For the first one we have that \eref{e:product2} holds by combining \eref{e:subset1} and \lref{p:start}, for which we have that
\begin{align}
  \lim_{\delta \to 0} \liminf_{v \to \infty} \frac1v \log\PP\left[X^v(t) \in \BB_{\dpp}(t_\delta w_j)\,,\, \Xi_j(\nvp,v)\,,\,\sigma_{\nvp,n_j}> t_\delta \middle| X^v(0) = x_0^v \right] = 0 \,,\label{e:ubonrf}
\end{align}
It remains to show that the second term is bounded by the rate function as in \eref{e:product2}. We first bound this term as
\begin{multline*}
  \inf_{y \in \BB_{\dpp}( x_0+t_\delta w)} \p{X^v \in \BBp_{[t_\delta, T]}(\delta^\prime, \rrr_\delta)~|~X^v(t_\delta)=y } \geq  \\
  \inf_{a \in \BB_{\dpp}( 0)} \p{X^v \in \BBp_{[t_\delta, T]}(\delta^\prime/2, \rrr_\delta+a)~|~X^v(t_\delta)=\rrr_\delta(t_\delta) + a },
  \end{multline*}
  where we shift the path $\rrr_\delta$ of $a$, but the lower bound is preserved since $\BBp_{[t_\delta, T]}(\delta^\prime/2, \rrr_\delta+a)\subseteq \BBp_{[t_\delta, T]}(\delta^\prime, \rrr_\delta)$ for all $a \in \BB_{\dpp}( 0)$.
  Since paths in the \abbr{rhs} above are uniformly bounded away from $\partial \Scal$ by \lref{l:breakup}, rates are uniformly bounded away from $0$ on the paths of interest and standard large-deviation bounds (which hold uniformly on $y \in \BB_{\dpp}(t_\delta \, w_{i_0} )$) can be applied. Therefore, defining $\mathcal N_\delta : = \BB_{\dpp}(0) \cap \mathrm{span}_{r \in \R}(\gamma^r)$ we bound the second term of \eref{e:product} by
\begin{align}
  \liminf_{v \to \infty}\frac1v \log \inf_{a \in \mathcal N_\delta}\p{X^v \in \BBp_{[t_\delta, T]}(\delta'/2, \rrr_\delta+a) \,|\,X^v(t_\delta) = \rrr_\delta(t_\delta) + a} &\geq
  \inf_{a \in \mathcal N_\delta} (-  \inf_{z \in \BBp_{[t_\delta, T]}(\delta'/2, \rrr_\delta+a)} I_{[t_\delta, T]}(z))\notag \\
  &\geq
  \inf_{a \in \mathcal N_\delta} (-  I_{[t_\delta, T]}(\rrr_\delta+a))\notag\\&
\geq  \inf_{a \in \mathcal N_\delta} (-  I_{[0, T]}(\tilde \rrr_\delta^a)) \,.
\end{align}
Finally, combining \lref{l:43} with the bound obtained above we obtain that
\begin{equ}
  \liminf_{\delta\to 0} \inf_{a \in \mathcal N_\delta} (-  I_{[0, T]}(\tilde \rrr_\delta^a))\geq - I_{[0, T]}(\rr)=- I^{x_0}_{[0, T]}(\rr)\,.\label{e:lbdone}
\end{equ}
\end{proof}

}

We conclude this section by proving \lref{l:breakup}
\begin{proof}[Proof of \lref{l:breakup}]
  %By exponential tightness, the condition $I_{[0,T]}(\rr) \leq K$ implies that there exists a compact $\mathcal K$ such that $\rr(t) \in \mathcal K$ for all $t \in [0,T]$. This in turn implies that the rates $\lambda_r$ are uniformly bounded on $\rr$. As a consequence, the only missing assumption of \cite[Lemma 3.5]{ShwartzWeiss05} in our statement is satisfied by exponential tightness, and the
  The finiteness of $J$ results from \cite[Lemma 3.5]{ShwartzWeiss05}.  {In particular, one can choose $\alpha$ small enough and $[\tau_{i-1},\tau_i]$ so that the set $\{\BB_\alpha(\rr(t))~,~t \in [\tau_{i-1},\tau_i]\}$ is contained in $\Acal_{j_i}$ for all $i \in (1, \dots, J)$.} By exponential tightness we can apply \cite[Lemma 3.4]{ShwartzWeiss05} to obtain  absolute continuity of $\rr$ on the set $I_{[0,T]}(\rr)<K$, so that there exists $\tau_- > 0$ with $\inf_{\{\rr~:~I_{[0,T]}(\rr)<K, i\in \mathbb N\}} \tau_{i}- \tau_{i-1}> \tau_-$. Consequently $J = T/\tau_-$ is finite.

  {
  The bound $\sup_{[0,T]} \|\rr-\rrr_\delta\| <  2\delta/3$ follows from the construction \eref{e:rtilde} and our choice of $\xi:= \min(1,(\kappa''/3)^{J+1}/3, \epsilon)$. Indeed for the time interval $[0, t_\delta]$ we have
  \begin{equ}
    \sup_{t \in [0, t_\delta]} \|x_0+w_j t -\rr(t)\| \leq t_\delta  + \omega_\rr(t_\delta) \leq \frac{2\delta}3\,,
  \end{equ}
  {where $\omega_\rr$} is the (subadditive) modulus of continuity of $\rr$.}
To extend this estimate beyond $t_\delta$ we note that
  \begin{equ}\label{e:distance}
    \sup_{t\in [0,T]}\|\rrr_\delta(t)-\rr(t)\| <   \sum_{k = 0}^J \beta^k t_\delta + \omega_\rr(\beta^k t_\delta) \leq 2 \xi \delta \sum_{k = 0}^J \beta^k \,.
    \end{equ}
    Then,  by our choice $\beta = 3/\kappa''$ and since
    \begin{equ}
        \sum_{l=0}^k \pc{\frac{3}{\kappa''}}^l =  \frac{1-(3/\kappa'')^{k+1}}{1-3/\kappa''} \leq  \frac{\kappa''}2(3/\kappa'')^{k+1} \,.\label{e:totalshift}
    \end{equ}
    we see that by boundedness of $k\leq J$ and by the definition of $\xi$ %%%\aanote{check no need for $\xi$ here}
    %we can choose $\xi > 0$ small enough that
    we have $\sup_{t\in [0,T]}\|\rrr_\delta(t)-\rr(t)\|< 2 \delta/3$.

% We now proceed to show that our choice of parameters implies that
% %$\mathcal B_{\kappa_- t}(x_0+w_j t) \cap \partial \mathcal S = \emptyset$  for all $t \in [0,\tau]$ for $\tau$ small enough.
%       the path $\rrr_\delta$ stays within the set $\SS$ for all times $t \in [0,T]$. Indeed, because $3/\kappa'' > 3$ we have by \eref{e:totalshift} that the cumulative shift obtained in the first $k$ transitions is smaller in absolute value than
%       $t_\delta  \sum_{l=0}^k \pc{{3}/{\kappa''}}^l \leq t_\delta \frac{\kappa''}2(3/\kappa'')^{k+1}$.

We now prove that $\bigcup_{t \in [t_\delta, T]}\mathcal B_{\kappa_- t_\delta}(\rrr_\delta(t)) \cap \partial \Scal = \emptyset$ by induction on $k$. For $k = 0$ the claim follows directly by \aref{a:escape}a) for $\delta<\epsilon'/2$. Then, by \eref{e:totalshift} as the $k+1$-th shift is of length $t_\delta (3/\kappa'')^{k+1}$ we must have that
     $\rrr_\delta(t)
     %+ w_{i_{k+1}} t_\delta (3/\kappa'')^{k+1}
     $  is at least at distance $t_\delta \kappa_- (1-\kappa''/2)  (3/\kappa'')^{k+1} > t_\delta \kappa_-/2   (3/\kappa'')^{k+1}>  \kappa_- t_\delta$ from $\partial \mathcal S$ for
     $t \in [\tau_{k+1}+ \Delta_{k+1}^\beta,\tau_{k+2}+ \Delta_{k+1}^\beta]$.
     Since the initial point of the shift satisfies the required condition by assumption, and that this property is conserved on $[\tau_k + \Delta_k^\beta, \tau_k + \Delta_{k+1}^\beta]$ (\ie during a shift) by \aref{a:escape}a) we obtain the desired result.

     Finally, we show that for every  $k\in (1,\dots,J)$ and $a\in \Rr^d$ with $\|a\|<t_\delta \kappa''/2$ there exists $\tilde w\in\mathcal B_{\kappa''}(0)$ such that $\rrr_\delta(\tau_k + \Delta_{k+1}^\beta)+a = z(\tau_k) + \beta^k t_\delta (w_{i_{k}}+ \tilde w)$. %recalling that $\delta'' = t_\delta \kappa''$.
     This follows immediately from \eref{e:totalshift}, since we have
     \begin{equs}
      \|\rrr_\delta(\tau_k + \Delta_{k+1}^\beta)+a - z(\tau_k) - \beta^{k} t_\delta w_{i_{k}}\|&=\|\sum_{l=0}^{k-1} \pc{\frac{3}{\kappa''}}^l t_\delta w_{i_l}  + a\| \leq \sum_{l=0}^{k-1} \pc{\frac{3}{\kappa''}}^l t_\delta  + \|a\| \\&\leq
      \frac{t_\delta \kappa''}2\pc{1+(3/\kappa'')^{k}} \leq \kappa''\|(3/\kappa'')^{k} t_\delta w_{i_{k}}\|\,,
    \end{equs}
    concluding the proof of the lemma.
\end{proof}

\section{LDP upper bound}\label{Sec:upper}\label{s:ub}
%\section{Optimality of \aref{a:escape}}

Similar results under slightly more restrictive assumptions are well known \eg \cite{dupuis91, ShwartzWeiss95} with jump rates bounded away from 0.
A sufficiently general result is available in \cite{pattersonrenger19}, but under assumptions on the initial condition that are not satisfied here.
We will sketch the application of the ideas from \cite{pattersonrenger19} to the setting of this paper.

In order to prove the upper bound, we will temporarily enlarge the state space in order to include the integrated flux of each reaction, i.\,e. we consider the process $(X^v(t),W^v(t))\in \RR^d\times\RR_{\geq0}^{|\Rcal|}$ with initial condition $(x_0^v,0)$ and generator:
\[
Q^vf(x,w)=\sum_{r\in\Rcal}v\Lambda^{v}_r(x)(f(x+v^{-1} {\gamma^r}, w+v^{-1} {\delta^r})-f(x,w)),
\]
for $\delta^r_r=1$ and $\delta^r_s=0$ for all $s\neq r$.
It is clear that the marginal distribution of the $X^v$-coordinate  is the distribution of our original process. In the following proposition, we prove a large-deviation upper bound for this process. To shorten notation, let us define for any $w\in\RR^{|\Rcal|}$ the vector $\Gamma w\colon =\sum_{r\in\Rcal}\gamma^r w_r\in \RR^d$.

Let $C^1_\mathrm{c}([0,T);\RR^{|\Rcal|})$ be the space of continuous and differentiable compactly supported functions from $[0,T)$ to $\RR^{|\Rcal|}$. For $x \in D_u(0,T;\S)$, $w \in D_u(0,T;\RR_{\geq0}^{|\Rcal|})$
%\lanote{What happened to $\BV$ space and topology?}
and $\zeta \in C^1_\mathrm{c}([0,T);\RR^{|\Rcal|})$ we set
\begin{equation*}
G(x,w,\zeta):= - \int_0^T \sum_{r\in\Rcal} \left(\dot \zeta_r(t) w_r(t) + \left[e^{\zeta_r(t)} - 1\right]\lambda_r(x(t))\right)\dd t
\end{equation*}
and use this to define a partial rate function
\begin{equation*}
\widetilde \Jcal_{\Scal}(x,w) :=
\begin{cases}
  \sup_{\zeta \in C_\mathrm{c}^1\left([0,T); \RR^\Rcal\right)} G(x,w, \zeta) & \text{ if } x(t) = x_0 + \Gamma w(t),  \qquad x(t)\in \Scal \quad \forall t \in [0,T) \\
  +\infty & \text{ otherwise.}
\end{cases}
\end{equation*}

\begin{proposition}\label{p:G-cont}
Let $\zeta \in C^1_\mathrm{c}([0,T);\RR^{|\Rcal|})$ and suppose Assumption~\ref{a:rate-converg} holds, then $(x,w) \mapsto G(x,w,\zeta)$ is continuous from $D_u(0,T; \Rr^d \times \RR_{\geq0}^{|\Rcal|})$  to $\RR$.
\end{proposition}
\begin{proof}
We have
\begin{equation*}
\abs{G(x,w,\zeta) - G(x^\prime, w^\prime, \zeta)}
\leq
\norm{\dot \zeta}_\infty \norm{w - w^\prime}_{\infty}T
 + \left(\e^{\norm{\zeta}_\infty} +1 \right) \int_0^T \sum_{r\in\R} \abs{\lambda_r(x(t)) -  \lambda_r(x^\prime(t))} \dd t.
\end{equation*}
%Now since Skorohod J1 convergence implies convergence almost everywhere
We can use the continuity of the $\lambda_r$ from Assumption~\ref{a:rate-converg}a), along with the boundedness given by the compactness of $\Scal$, and apply dominated convergence when $x^\prime \to x$ to see that the second term of our estimate vanishes, as well as the first term when $w^\prime\to w$.
%Since Skorohod J1 convergence implies L1-norm convergence we see that for any convergent sequence $G(\cdot,\cdot, \zeta)$ also converges and is thus continuous.
\end{proof}

\begin{proposition}\label{p:upper-bound}
Let $\Kcal$ be a closed subset of the space of càdlàg paths $ D_u(0,T;\Rr^d\times\RR_{\geq0}^{|\Rcal|})$, then under Assumption~\ref{a:rate-converg}
%\lanote{Ass 1 is the only needed, right?}
\begin{equation*}
\limsup_{v \to \infty} \frac1v \log \PP\left((X^v,W^v) \in \Kcal \right) \leq -\inf_{(x,w)\in\Kcal} \widetilde{\Jcal}_{\Scal}(x,w).
\end{equation*}
\end{proposition}

\begin{proof}
Assumption~\ref{a:rate-converg} implies exponential tightness -- see discussion below the statement of \tref{th:main} -- therefore we may assume that $\Kcal$ is compact.

Fix $\epsilon \in (0,1)$, then for every $(x,w) \in D_u(0,T;\Scal\times\RR_{\geq0}^{|\Rcal|})$ satisfying $x = x_0 + \Gamma w$ one can find $\zeta[x,w] \in C_\mathrm{c}^1\left([0,T); \RR^\Rcal\right)$ such that
\begin{equation*}
G\left(x,w, \zeta[x,w]\right) \geq \min\left(\widetilde{\Jcal}(x,w), \epsilon^{-1}\right) - \epsilon
\end{equation*}
and define neighbourhoods in path space
\begin{equation*}
\Gcal_{\epsilon}(x,w) := \left\{(x^\prime, w^\prime) : G(x^\prime, w^\prime, \zeta[x,w]) \geq G(x,w, \zeta[x,w]) -\epsilon\right\}.
\end{equation*}
We may use Proposition~\ref{p:G-cont} to see that the $\Gcal_{\epsilon}(x,w)$ are open and so can find a finite cover $\Gcal_{\epsilon}(x^i,w^i)\ i=1,\dotsc, n$ for $\Kcal$.

Now following \cite[Thm A.3]{pattersonrenger19} and using the fact that the jump rates are bounded over $\Scal$ (because of \aref{a:rate-converg}a) and compactness of $\Scal$),
we define tilted measures $\PP^{\zeta}$ via mean 1 non-negative martingales so that
\begin{equation*}
\frac1v\log \frac{\d\,\PP^\zeta\circ\left(X^v,W^v\right)^{-1}}{\d\,\PP\circ\left(X^v,W^v\right)^{-1}}(x,w)=\\
{-\int_0^T\sum_{r\in\Rcal} w_r(t) \dot \zeta_r(t)+\Lambda^{v}_r(x(t))\left(\e^{ \zeta_r(t)}-1\right)\d t}:=G^v(x,w,\zeta).
\end{equation*}
Slightly adapting and simplifying the argument of  \cite[Lemma 4.7]{pattersonrenger19}
\begin{multline}
\frac1v \log \PP\left((X^v, W^v) \in \Gcal_{\epsilon}(x,w)\right) \\
\leq
\frac1v \log \PP^{\zeta[x,w]}\left((X^v, W^v) \in \Gcal_{\epsilon}(x,w)\right)
 - \inf_{(x^\prime, w^\prime) \in \Gcal_{\epsilon}(x,w)} G^v(x^\prime, w^\prime, \zeta[x,w]) \\
 \leq  -\inf_{(x^\prime, w^\prime) \in \Gcal_{\epsilon}(x,w)} \abs{G^v(x^\prime, w^\prime, \zeta[x,w])-G(x^\prime, w^\prime, \zeta[x,w]) }
 -  \inf_{(x^\prime, w^\prime) \in \Gcal_{\epsilon}(x,w)} G(x^\prime, w^\prime, \zeta[x,w]).
\end{multline}
The first term on the final line vanishes as $v \to \infty$ by the uniform convergence from Assumption~\ref{a:rate-converg}b).
Thus for $i =1, \dotsc, n$
\begin{equation*}
\frac1v \log \PP\left((X^v, W^v) \in \Gcal_{\epsilon}(x^i,w^i)\right)
\leq
-  \min\left(\widetilde{\Jcal}(x^i,w^i), \epsilon^{-1}\right) - 2\epsilon
\end{equation*}
and by the Laplace principle
\begin{equation*}
\frac1v \log \PP\left((X^v, W^v) \in \Kcal\right)
\leq
- \min_{i=1,\dotsc, n}  \min\left(\widetilde{\Jcal}(x^i,w^i), \epsilon^{-1}\right) - 2\epsilon
\leq \inf_{(x^\prime, w^\prime) \in \Kcal} \min\left(\Jcal(x^\prime, w^\prime), \epsilon^{-1}\right)  - 2\epsilon,
\end{equation*}
which completes the proof as $\epsilon$ can be taken arbitrarily small.

If $\nexists$ $(x,w) \in \Kcal$ satisfying $x = x_0 + \Gamma w$, then $\limsup_{v \to \infty} \frac1v \log \PP\left((X^v,W^v) \in \Kcal \right) \leq-\infty$, since, by definition  $X^v(t)=X^v(0)+\Gamma W^v(t)$ a.s. for all $t\in[0,T]$.
\end{proof}

%This upper bound may be made sharper in the following obvious way:\begin{corollary}Let $\Kcal$ be a closed subset of $ D_s(0,T;\S\times\RR_{\geq0}^{|\Rcal|})$, let Assumptions~\ref{a:rate-converg} hold. Further suppose that $A$ is a closed set such that$\limsup_v \frac1v \log \PP\left((X^v,W^v) \in A\right) = - \infty$, then\begin{equation*}\limsup_v \frac1v \log \PP\left((X^v,W^v) \in \Kcal \right)\leq-\inf_{(x,w)\in\Kcal} \widetilde{\Jcal}_A(x,w),\end{equation*}where $\widetilde{\Jcal}_A(x,w):=\widetilde{\Jcal}(x,w) + \chi_A(x,w)$, with $\chi_A$ is $+\infty$ outside $A$ and $0$ inside.\end{corollary}In particular we may apply this corollary with $A = \S$.

\begin{proposition}\label{p:equivalence}
If Assumption~\ref{a:rate-converg} holds %\rp{(we do not need the part with $\aleph$)},
then $\Jcal=\widetilde\Jcal_\Scal$, where
%for any path $(x,w)$ such that $\widetilde{\Jcal}(x,w)<\infty$
\begin{equation*}
{\Jcal}(x,w) :=
  \begin{cases}
    \sum_{r\in\Rcal} \int_0^T\! \Hcal(\dot w(t)| \lambda(x(t)))\,\dd t, &(x,w)\in \Acal\Ccal(0,T;\Scal\times\RR^{|\R|}), \dot x=\Gamma \dot w, x(0)=x_0\\
    \infty,                                                        &\text{otherwise}.
  \end{cases}
\end{equation*}
where $\Hcal$ is defined in \eqref{e:lbrf2}. %{Moreover, if $(x,w)\notin \Acal\Ccal(0,T;\Scal_*\times \RR_{\geq0}^{|\Rcal|})$ then $\widetilde{\Jcal}(x,w) =\infty$.}
\end{proposition}
\begin{proof}
The proof follows \cite[Prop 3.5]{pattersonrenger19}; we assume that $\dot x=\Gamma \dot w$ and $x(t)\in \Scal$ a.e. $t\in[0,T)$ throughout.
%with the exception that we must first establish that $\widetilde{\Jcal}(c,w) = + \infty$ unless $w$ (and thus also $c$) is a function of bounded variation \cite{??,HeidaPattersonRenger2019} \rp{Michiel please add a sensible reference here} .
Suppose that $(x,w)\notin \Acal\Ccal(0,T;\Scal\times \RR_{\geq0}^{|\Rcal|})$, then one can find a sequence $\zeta^n \in C_\mathrm{c}^1\left([0,T); \RR^{|\Rcal|}\right)$ with $\sup_{n,\rr,t} \abs{\zeta^n_\rr(t)} \leq 1$ but
$\lim_{n \rightarrow \infty} \int_0^T \sum_{r\in\Rcal}\dot \zeta^n_\rr(t) w(t) \dd t = -\infty$
and thus $\Jcal (x,w) \geq \lim_{n \rightarrow \infty}G(c,w,\zeta^n) = + \infty$,
that is, $\Jcal =\tilde\Jcal_{\Scal}$ if the path is not absolutely continuous.

%Now take any $(x,w)\in \Acal\Ccal(0,T;\Scal_*\times \RR_{\geq0}^{|\Rcal|})$. If $(x,w)\notin\W^{1,1}(0,T;\RR^d\times\RR^\R)$ then by Jensen's inequality also $\sum_{r\in\Rcal} \int_0^T\! f(\dot w(t), \lambda(x(t)))\,\dd t=\infty$ so we may still use the same formula. Clearly,
%\begin{align*}
%  \widetilde\J(x,w)=\sup_{\theta \in C_\mathrm{b}^1\left((0,T); \RR^d\right)}\int_0^T\pq{\theta_s \cdot \dot x_s - \sum_{r \in \R} \lambda_r(x_s) \pc{\exp\pc{\theta_s \cdot \gamma_r}-1}}\dd s
%\end{align*}

One then shows that $\Jcal (x,w)=\widetilde{\Jcal}_{\Scal}(x,w)$ for any $(x,w)\in\AC(0,T;\Scal\times\RR^{|\R|})$ using approximation arguments.
\end{proof}

\begin{corollary}\label{cor:upp_bound}
If Assumption~\ref{a:rate-converg} holds then the large-deviation upper bound holds with good rate functional:
\begin{align*}
  \inf_{\substack{w\in W^{1,1}(0,T;\RR^{|\R|}_{\geq0}) \\ \dot x=\Gamma \dot w}}{\Jcal}(x,w)=I^{x_0}_{[0,T]}(x)=
  \int_0^T\! \sup_{\vartheta\in \RR^d}\Big\lbrack \vartheta\cdot\dot x(t) - \sum_{r\in\R}\lambda_r(x(t))\big(\exp(\vartheta\cdot\gamma^r)-1\big)\Big\rbrack\,\d t.
  %\sup_{\theta \in C_\mathrm{b}^1\left((0,T); \RR^d\right)}\int_0^T\pq{\theta_s \cdot \dot x_s - \sum_{r \in \R} \lambda_r(x_s) \pc{\exp\pc{\theta_s \cdot \gamma_r}-1}}\dd s.
\end{align*}
\end{corollary}
\begin{proof}
  The large-deviation upper bound with the rate functional on the left-hand side follows from Propositions~\ref{p:upper-bound}, \ref{p:equivalence} and the contraction principle. For non-absolutely continuous paths both the left-hand and right-hand sides will diverge: The left-hand side since the infimum will be taken over an empty set, and the right-hand by a similar argument as in Proposition~\ref{p:equivalence}.

Now take an arbitrary $x\in \Acal\Ccal(0,T;\RR^d)$. Then
\begin{align*}
  \inf_{\substack{w\in W^{1,1}(0,T;\RR^{|\R|}_{\geq0}) \\ \dot x=\Gamma \dot w}}{\Jcal}(x,w)
    &\geq
  \underbrace{\int_0^T\! \inf_{j\in \RR^{|\R|}_{\geq0}: \dot x(t)=\Gamma j} \Hcal\big(j(t)|\lambda(x(t))\big)\,\d t}_{=I^{x_0}_{[0,T]}(x)}\\
    &=
  \int_0^T\! \sup_{\vartheta\in \RR^d}\Big\lbrack \vartheta\cdot\dot x(t) - \sum_{r\in\R}\lambda_r(x(t))\big(\exp(\vartheta\cdot\gamma^r)-1\big)\Big\rbrack\,\d t,\\
  %  &\geq\sup_{\vartheta\in C^1_b((0,T);\RR^d)} \int_0^T\!\Big\lbrack \vartheta(t)\cdot\dot x(t) - \sum_{r\in\R}\lambda_r(x(t))\big(\exp(\vartheta(t)\cdot\gamma^r)-1\big)\Big\rbrack\,\d t,
\end{align*}
where the equality follows from convex duality, pointwise in $t$. To show that the inequality is in fact an equality, we may assume that the left-hand side is finite.
%\lanote{Check the case in which it is not finite.}
Hence from now on we may assume that $x\in W^{1,1}(0,T;\Scal)$. By Jensen's inequality, any path $j:(0,T)\to\RR^{|\R|}_{\geq0}$ for which $\int_0^T\!\Hcal\big(j(t)|\lambda(x(t))\big)\,\d t<\infty$ is bounded in $L^1(0,T;\RR^{|\R|})$, which shows the first equality.

%To show that the second inequality above is also an equality, one can use the following two standard approximation arguments. Take any path $\vartheta:(0,T)\to\RR^d$ such that $\int_0^T\!\vartheta(t)\cdot\dot x(t) - \sum_{r\in\R}\lambda_r(x(t)) \big(\exp(\vartheta(t)\cdot\gamma_r)-1\big)\,dt<\infty$ and approximate $\vartheta^n(t)=\vartheta(t)\mathds1_{[-n,n]}(\vartheta(t))$. The integral then converges as $n\to\infty$ by monotone convergence on the positive and negative parts of the integrand, which shows that the supremum can be taken over bounded functions. Now take such a bounded function $\vartheta\in L^\infty(0,T;\RR^\R)$ and approximate it by convoluting with a heat kernel. Then by dominating convergence the integral converges, which shows that the supremum can be taken over $C^\infty_b(0,T;\RR^\R)$.
\end{proof}

\section{Optimality of decay rate in \eref{e:decay}}\label{Sec:counterex}\

 %The \abbr{LDP} upper bound in the previous section is obtained without using the integrability of the rates from \aref{a:escape}b).
 We recall from \rref{r:decay} that integrability of the rates necessary to establish the lower bound estimates in Section~\ref{s:lb} -- but not the upper bound ones in Section~\ref{s:ub} -- is directly implied by a sufficiently slow decay of the rates \eref{e:decay}.
 %is indeed crucial to have a proper \abbr{LDP} on scale $v$, by presenting a large class of examples where~\aref{a:escape}b)
 In this section (and more specifically in \pref{p:optimality}) we make precise our claim that the range of exponents $\alpha$ given in \rref{r:decay} is maximal. In particular, we show that whenever the rates of jumps necessary to escape the degenerate set decay too fast (satisfying a condition similar to \eref{e:decay} for $\alpha \geq 1$), the rate function for the upper bound diverges for any $y\in  \Acal \Ccal(0,T;\Scal)$ with $y(0) \in \partial \mathcal S$ and
 $y(t) \in \Scal\setminus \partial \Scal$ for a $t>0$. We start the discussion of this problem with some examples as to capture the idea of our strategy in a simple setting.
 %, given any non empty subspace $\Scal_*\subset \Scal$ and $x(0)\in \partial \Scal_*$.
Recall from Corollary~\ref{cor:upp_bound} that, under \aref{a:rate-converg}, for any $y\in  \Acal \Ccal(0,T;\Scal)$ and any $\delta>0$:
 \begin{equation}\label{eq:ell_upper}
\limsup\limits_{v\to\infty}\frac 1v \log \PP_{x_0^v}(X^v\in \overline{\BBp_{[0,T]}(\delta,y)})\leq -\inf_{\substack{x\in \overline{\BBp_{[0,T]}(\delta,y)}}}
%\sup_{\theta \in C_\mathrm{b}^1\left((0,T); \RR^d\right)}
\int_0^T\ell(x(s),\dot x(s))\dd s, \end{equation}
where we recall that $\ell$ is defined as
\begin{equ}
\ell(x,y)=\sup_{\theta\in \RR^{d}}{\theta \cdot y - \sum_{r \in \R} \lambda_r(x) \pc{\exp\pc{\theta \cdot \gamma_r}-1}}.
\end{equ}
\begin{example}
Recall Example~\ref{ex:diverging1d}, where \eref{e:decay} does not hold and the upper bound is easily seen to diverge.  Given the generator in \eqref{eq:example},
% We consider a jump model with the following generator:\begin{equ}  \LL^vf(x) = ve^{-1/x}(f(x+v^{-1})-f(x))\,,\end{equ}
the integrand on the right-hand side of the \abbr{LDP} upper bound in  \eqref{eq:ell_upper} reads
\begin{equ}
  \ell(x,y) = \sup_{\theta\in\RR^d}\big\{ \theta y - e^{-1/x} (e^{\theta}-1)\big\} \geq y\log ye^{1/x} - (y-e^{-1/x}),
\end{equ}
where in the inequality we have chosen $\theta(x,y) = \log ye^{1/x}$.
Then, the \abbr{LDP} rate function can be bounded as follows
\begin{equ}
  I^{0}_{[0,T]}(z) \geq \int_0^T z'(t)\log (z'(t) e^{1/z(t)}) - (z'(t)-e^{-1/z(t)}) \d t \geq \int_0^T z'(t)\log z'(t) + z'(t)/z(t) - z(t) \,\d t,
\end{equ}
for any $z\in \Acal \Ccal (0,T;\RR_{\geq0})$ with $z(0)=0$.
Using that $x\log x> -1$ is continuous at 0 we have, for every bounded path $z$ with bounded derivative, that
\begin{equ}
  I^0_{[0,T]}(z) \geq -1 - z(1) + \int_0^1 z'(t)/z(t) \,\d t  \geq -1 - z(T) + \int_0^{z(T)} x^{-1}\,\d x = +\infty.
\end{equ}
\end{example}

In order to proceed and generalize the example above, we discuss two further examples highlighting the appropriate way to negate the assumptions in \rref{r:decay}.
\begin{example}
  Consider a system defined on $\mathcal S_v = (v^{-1}\mathbb N_0)^2$ with two jumps, $\gamma_1 = (0,1)$, $\gamma_2 = (1,0)$ and corresponding rates $\Lambda_1^v(x) = \lambda_1(x) = \one\{x_1\geq 1\}(x_1-1) $, $\Lambda_2^v(x) = \lambda_2(x) = 1$. For the initial condition $x_0 = 0$ the law of large numbers \cite{Kurtz1972} shows that the paths of $X^v$ concentrate around $(y_1^*, y_2^*)(t) = (t, \one\{t>1\} (t-1)^2/2)$. In particular, this implies the existence of paths $y(t) \in \Acal \C(0,T;\mathcal S)$ with $y(0) \in \partial \mathcal S$ and $y(t) \in \Scal\setminus \partial \Scal$  but having a finite large deviations cost for a system violating \eref{e:decay}.
\end{example}
To discuss the optimality of the interval for the parameter $\alpha$ governing the \emph{local} decay of rates in \eref{e:decay} we avoid  considering macroscopic behaviors like the one highlighted in the example above and we restrict our attention to paths that do not leave the set $\mathcal A_j$ in the time interval of interest, keeping $j$ fixed throughout this section. For the same reason, to negate \eref{e:decay} we consider jumps whose rates decay faster than $\exp[-k\cdot \mathrm{dist}(x,\partial\Acal_j)^{-1}]$ \emph{uniformly} in $x$ in $\mathcal A_j$ for a $k>0$. These jumps belong to the set
\begin{equ}\label{e:fast}
\mathrm{FAST}_{\Rcal,j}:=\pg{r \in\Rcal~: \lim_{\rho \to 0}\rho \pc{\sup_{z \in \Acal_j \colon \mathrm{dist}(z,\partial\Acal_j)< \rho} \log\lambda_r(z)} <0},
\end{equ}
with $\mathrm{dist}(z,\partial\Acal_j)=\inf_{x\in\partial \Acal_j}\|z-x\|$.
While this set is not, in general, the complement of the jumps whose rates satisfy \eref{e:decay}, it allows to capture, at least locally, those whose decay is more rapid (in terms of $\alpha$) than \eref{e:decay}.

We further notice that the existence of a single reaction $r \in \mathrm{FAST}_\R$ may still result in a finite cost for paths escaping $\partial \mathcal A_j$, as the following example shows.
\begin{example}
  Consider a system defined on $\mathcal S_v = (v^{-1}\mathbb N_0)^2$ with two jumps, $\gamma_1 = (-1,1)$ and $\gamma_2 = (1,1)$ and corresponding rates $\lambda_1(x) = x_1 e^{-1/x_2}$, $\lambda_2(x) = 1$. It is clear that this system satisfies a \abbr{LDP} with any sequence of initial conditions, see \cite{ShwartzWeiss05}. However, approaching the set $\{x\in \Rr^2~:~x_2 = 0\}$, upon choosing $w_j = (1,0)$ we see that  this system does not satisfy \eref{e:decay}. We are, however, able to \emph{choose} $w_j = (1,1)$ so that \emph{with such choice of} $w_j$ \eref{e:decay} holds for all $r \in \mathcal E_j$.
\end{example}

{In light of the above example, the statement we seek to negate is the existence of vectors $w_j$ and corresponding $\mathcal E_j$ stated in \aref{a:escape} such that $\mathcal E_j \cap \mathrm{FAST}_{\Rcal,j} = \emptyset$.}
%\eref{e:decay} holds for all $r \in \mathcal E_j$.}
To do so, we fix a
%convergent sequence of initial conditions $\{x_0^v\}$, which implies the existence of a
(non-empty) $\S$ and a limiting point $x_0\in \partial \Acal_j$, for $j\in \J$, assuming throughout that $\partial \Acal_j$ is a subset of a $(d-1)$-dimensional hyperplane to simplify the notation of the proof.
In this way, we define {$T_x\Acal_j=\{y\in \mathrm{span}_{r \in \R}(\gamma^r)\colon y\cdot n_x\geq 0 \}$}, with $n_x$ inward normal to the boundary $\partial \Acal_j$ in $x$. Assuming that there is no vector {$w_j\in \mathrm{span}_{r \in \R}(\gamma^r)$} that is a sum of jumps with rates decaying slow enough means that $\forall x\in \partial \Acal_j$
\begin{equation}\label{eq:bad_reactions}
  \mathrm{Co}_{x}\pc{\{\gamma^r~:r \in \pc{\R \setminus \mathrm{FAST}_{\Rcal,j}}\}} \cap T_{x}\Acal_j = \emptyset\,,
\end{equation}
where $\text{Co}_x(A)$ is the convex cone defined by the set of vectors $A$ with origin $x$.

Note that in this way we are building a class of processes where the jumps $r$ pointing in the interior of the domain (and therefore useful to escape the boundary) necessarily belong  to  $ \mathrm{FAST}_{\Rcal,j}$.
% and do not satisfy the condition
% \begin{equation}\label{eq:bad_rates}
% \lim_{\rho \to 0}\rho \pc{\inf_{z \in \Acal_j \colon \mathrm{dist}(z,\partial\Acal_j)> \rho} \log\lambda_r(z)} =0.
% \end{equation}
% \aa{Indeed,  in the above definition we do not consider rates decaying like $e^{-(\rho \log \rho)^{-1}}$ as $\rho \to 0$, which satisfy \eqref{eq:bad_rates} but not Assumption \ref{a:escape}a). This inaccuracy, however, allows to significantly simplify the argument for the proof of the divergence of the rate functional in this setting. Since this is not the main purpose of this paper, we content ourselves with proving the desired result in this setting and sketching how this can be generalized to the rates not covered by our setting.  We should also stress that we are not aware of any application where such rates are actually used.}\aanote{LUISA: we can now remove all the above in green?}\lanote{For me, the above can go away.}

\begin{proposition}\label{p:optimality}
  Assume that \eref{eq:bad_reactions} holds, then for every $y \in \Acal \C(0,T;\mathcal S)$ with $y(0) = x_0 \in \partial \Acal_j$  and such that there exists $t_1 \in (0,T)$ with $y(t_1) \in \Acal_j\setminus \partial \Acal_j$ and $\inf\{t \in (0,T)~:~z(t) \not \in \Acal_j\}>t_1$ it holds that $I_{[0,T]}^{x_0}(y) = \infty$.
\end{proposition}

\begin{proof}Recalling the structure in \eref{eq:ell_upper} of the \abbr{LDP} upper bound we notice that, in order to show that the rate function is infinite for any path $y \in \Acal \Ccal(0,T;\SS)$ as above,
  %whenever \eref{e:decay} is not satisfied,
  it is sufficient to find
  %, for every path $y\in \Acal \Ccal(0,T;\SS)$ with $y(0)=x_0\in \partial \SS$
  a $\theta(t,y)$ such that
\begin{equation}\label{eq:tuned_upper_buond}
\int_0^T\pq{\theta(s,y(s)) \cdot \dot y(s) - \sum_{r \in \R} \lambda_r(y(s)) \pc{\exp\pq{\theta(s,y(s)) \cdot \gamma_r}-1}}\dd s = +\infty\,.
\end{equation}
%We consider the simple choice of $\theta(y,\dot y) := \phi \log\frac{\phi \cdot \dot y}{\lambda(y)}$ for $\phi$ and $\lambda(y)$ to be defined below, so that\begin{equ}  \theta \cdot \dot y - \sum_{r \in \R} \lambda_r(y) \pc{\exp\pc{\theta \cdot \gamma_r}-1} = \dot y\cdot \phi \log\frac{\phi \cdot \dot y}{\lambda(y)} - \sum_{r \in \R} \lambda_r(y) \pc{\pc{\frac{\phi \cdot \dot y}{\lambda(y)}}^{\phi \cdot \gamma_r}-1}\end{equ} Hence, we can choose any $\phi=n_x$ ({or any other vector in $\Scal_*$}) and and $\lambda(y) = e^{-\kappa/{\phi \cdot y}}$ for $\kappa$ small enough and obtain{\bf Some other attempt:} To bound the fourth integral seems unrealistic unless we know something about $\dot y$. However, we don't need to include $\dot y$ in the test function $\vartheta$.
By assumption, there exists $t_0<t_1 \in [0,T]$ such that $y(t_0) \in \partial \mathcal A_j$ and $y(t) \not \in \partial \Acal_j$ for all $t\in (t_0,t_1)$.
%Without loss of generality, let $t_0=0$ (else let $\vartheta(t,y(t))= 0$ for $t\in[0,t_0]$).
We now aim to express $\int_{t_0}^{t_1}\!\vartheta(t,y(t))\cdot \dot y(t)\,\d t$ as an exact integral of some potential for which $\Phi(x_0)=-\infty$, %Without loss of generality we assume that $\|y(t)-x_0\|>0$ for all $t\in (0,T]$. Indeed, if $y(t_1)=x_0$ for some $t_1>0$ and $y(t)\neq x_0$ for all $t\in(0,t_1)$, then we consider the part of the integral in \eqref{eq:tuned_upper_buond}  between $0$ and $\frac {t_1}2$ and express $\int_0^{\frac{t_1}2}\!\vartheta_t\cdot \dot y(t)\,\d t$ as the integral of the potential $\Phi$.
 while we choose $\vartheta(t,y(t))=0$ for $t\in[0,t_0]\cup[t_1,T]$, such that the integral in \eqref{eq:tuned_upper_buond} vanishes on that interval.
%\la{We cannot ask $\Phi(y_0)=-\infty$ directly because of the conditions from Corollary \ref{cor:upp_bound}, where $\theta$ has to be bounded and with compact support in $[0,T)$.}\lanote{We got from Corollary \ref{cor:upp_bound} that the rate function can be obtained as a $\sup$ over $\theta \in C_\mathrm{c}^1\left([0,T); \RR^d\right)$, which means basically that $\theta(T)=0$ and $\theta(0)<+\infty$. This is not realized if we chose $\theta(y_s,\dot y_s)=\kappa\nabla\Phi(y_s)$ unless we tune properly the potential $\Phi$.  }
More specifically, following for example~\cite{HilderPeletierSharmaTse20}, for a $\kappa>0$ we take $\vartheta=\kappa \nabla\Phi(y)$ so that
\begin{equation*}
  I(y)\geq  \kappa\Phi(y(t_1)) - \kappa\Phi(y(t_0)) - \sum_{r \in \R} \int_{t_0}^{{t_1}}\!\lambda_r(y(t))e^{\kappa \gamma_r\cdot \nabla\Phi(y)}\,\d t + \underbrace{\sum_{r \in \R} \int_{t_0}^{t_1}\!\lambda_r(y(t))\,\d t}_{\geq0}\,.
\end{equation*}
The missing step is therefore to choose $\Phi$ and tune $\kappa$ such that $\Phi(y(t_0))=-\infty$ and $ \sum_{r \in \R} \int_{t_0}^{t_1}\!\lambda_r(y(t))e^{\kappa \gamma_r\cdot \nabla\Phi(y)}\,\d t $  is bounded. For the choice $\Phi(y):=\log(n_{x_0}\cdot (y-x_0))$ we have:
\begin{align*}
  \sum_{r \in \R} \int_{t_0}^{t_1}\!\lambda_r(y_t)\exp\pq{\kappa \gamma_r\cdot \nabla\Phi(y)}\,\d t & = \sum_{r \in \R} \int_{t_0}^{t_1}\!e^{\log\lambda_r(y(t))+\kappa (n_x\cdot \gamma_r)/(n_{x_0}\cdot (y(t)-x_0))}\,\d t\,.
%  &= \sum_r \int_0^T\!e^{\frac1{n_{x_0}\cdot (y(t)-x_0)}((n_{x_0}\cdot (y(t)-x_0))\log\lambda_r(y_t)+\kappa (n_x\cdot \gamma_r)))}\,dt.\\
 % &= \sum_{r\in\mathrm{SLOW}_{\Rcal}} \int_0^T\! \lambda_r(y(t)) e^{\kappa\tfrac{ n_x\cdot \gamma_r}{n_x\cdot (y(t)-x_0)}}\,dt + \sum_{r\in\Rcal\setminus \mathrm{SLOW}_{\Rcal}} \int_0^T\! \lambda_r(y(t)) e^{\kappa\tfrac{ n_x\cdot \gamma_r}{n_x\cdot (y(t)-x_0)}}\,dt.
\end{align*}
 With this choice, since $y\in \Acal\Ccal(0,T;\Scal)$, we have that $n_{x_0}\cdot (y(t)-x_0)\geq 0$ for $t\in (t_0,t_1)$. %We will restrict therefore to prove the boundedness of the above integrals between $0$ and $\epsilon$, since the boundedness between $\epsilon$ and $T$ is ensured by the fact that $\|y(t)-x_0\|\geq c_{\epsilon}>0$, for a certain constant $c_{\epsilon}$.
 We can split the jumps in the set $\Rcal_+:=\{r\in\Rcal:\gamma^r\in T_{x_0}\Acal_j\}$ and $\Rcal_-=\Rcal\setminus \Rcal_+$. For the latter class of jumps we have $n_{x_0}\cdot\gamma_r \leq 0$ and
 \[
 \sum_{r\in\Rcal_-} \int_{t_0}^{t_1}\! \lambda_r(y(t)) e^{\kappa\tfrac{ n_{x_0}\cdot \gamma_r}{n_{x_0}\cdot (y(t)-x_0)}}\,\d t<+\infty\,,
\]
 since the argument of each integral is bounded on $(t_0,t_1)$. On the other hand, we handle the terms coming from
the former class of jumps $\Rcal_+$ -- the ones pushing the process in the interior of $\Scal$ -- using that $\Rcal_+\subseteq  \mathrm{FAST}_{\Rcal,j}$\,.
Therefore
% such that $\lim_{\rho \to 0}\rho \pc{\inf_{z \in \BB_j \colon d(z,x)> \rho} \log\lambda_r(z)} <0$, is it true that
we  can tune $\kappa$ such that
\begin{equation*}
\lim_{t\rightarrow0}(n_{x_0} \cdot (y(t)-x_0)) \log\lambda_r(y_t)+\kappa (n_{x_0}\cdot \gamma_r)\leq 0\qquad \forall r\in \Rcal_+\,,
 \end{equation*}
ensuring that $ \sum_{r\in\Rcal_+} \int_{t_0}^{t_1}\! \lambda_r(y(t)) \exp\pq{\kappa\tfrac{ n_{x_0}\cdot \gamma_r}{n_{x_0}\cdot (y(t)-x_0)}}\,\d t<+\infty$. This proves \eqref{eq:tuned_upper_buond}.
%\la{I have the impression that rates of the form $\lambda_r(x)=\e^{\frac{1}{z\log z}}$ for $z=d(x,\partial \Bcal_j)$ will still not be suitable in this case, since one would have to choose $\kappa=0$, which is not possible. This seems to be saying that a condition like~\eqref{cond} is maybe more suitable than~\aref{a:escape}a), at least when $\lambda_r(x)$ only depends on $z$. }
\end{proof}

\section{List of symbols}\label{Sec:constants}
\phantom{a}
%\mr{(Please move this `section' to any place you find convenient, check/change the descriptions, add symbols if some important ones are still missing ($\aleph$? $\Xi$?), and change the order if desirable!)}
\begin{table}[htp]
\centering
\begin{tabular}{l l l}
  $\R$                     & finite set of jumps/reactions                            & Subsec.~\ref{subsec:intro ldps} \\
  $\Lambda^v_r,\lambda_r$  & microscopic and macroscopic jump rates                   & Subsec.~\ref{subsec:intro ldps} \& Ass.~\ref{a:rate-converg}, \ref{a:escape} \\
  $\gamma^r\in\RR^d$       & jump vectors                                             & Subsec.~\ref{subsec:intro ldps} \& Ass.~\ref{a:escape}\\
  $\SS_v,\SS,\partial\SS\subset\RR^d$  & reachable points and boundary/degenerate set & Sec.~\ref{Sec:not_res}\\
  $\Acal_i,\partial\Acal_i\subset\RR^d$ & covering of $\SS$                           & Sec.~\ref{Sec:not_res}\\
  $\I,\J$,                 & index sets for covering                                  & Sec.~\ref{Sec:not_res}\\
  $\BB_\rho(x)\subset\RR^d$& euclidean ball of radius $\rho$ and center $x$           & Subsec.~\ref{s:assumptions}\\
  $D_u(0,T;\RR^d) \;\big(D_s(0,T;\RR^d)\big)$ & \cadlag functions with uniform (Skorohod) topology &  Subsec.~\ref{subsec:the ldp}\\
  $\Acal \Ccal(0,T;\RR^d)\; \big(\Acal \Ccal(0,T;\Scal)\big)$ & absolutely continuous functions (restricted to $\Scal$) &  Subsec.~\ref{subsec:the ldp}\\
  $\BBp_{[0,T]}(\rho,z)$   & ball of radius $\rho$ and center $z$ in $D_u(0,T;\RR^d)$ &  Subsec.~\ref{subsec:the ldp}\\
  $I_{[0,T]}, I_{[0,T]}^{x_0}$ & large-deviation action and rate functional           & Eqns.~\eqref{e:lbrf}, \eqref{e:rf2}\\
  $\Hcal( \mu|\lambda), \ell(x,y) $ & relative entropy and Lagrangian                 & Eqn.~\eqref{e:lbrf2}, \eqref{e:ell}\\
  $\Jcal,\widetilde\Jcal_\Scal$      & flux large-deviation functional and dual form  & Section~\ref{s:ub}\\
  $\Ecal_j\subset\R, w_j\in\RR^d,\alpha_j>0$ & escape sequence, vector and normalisation & Ass.~\ref{a:escape} \ref{assit:wj}), \ref{assit:Ej})\\
  $\epsilon,\epsilon',\kappa_j,\kappa_->0, $    & escape parameters                     & Ass.~\ref{a:escape} \ref{assit:wj})\\
  $\epsilon'',\kappa''>0$  & monotonicity range                                       & Ass.~\ref{a:escape} \ref{assit:monotonicity})\\
  $\rr,\rrr_\delta$        & target and approximated path                             & Subsec.~\ref{s:path}\\
  $t_\delta,\xi,\beta>0,\omega_z$ & path shift parameters, modulus of continuity      & Eqns. \eqref{e:rtilde},\eqref{e:rtilde2}\\
    $\delta', \delta''>0$ & neighborhood parameters of shifted path& Subsec.~\ref{s:path}\\

\end{tabular}
\end{table}

\section*{Acknowledgements}
AA was partially supported by  the Swiss National Science Foundation grant P2GEP2-175015 and by the NSF grant DMS-1613337. He furhter acknowledges the hospitality of the Weierstrass Institute of Applied for Applied Analysis and Stochastics. LA acknowledges the hospitality of the Hausdorff Institute in Bonn, during the Junior trimester program
``Randomness, PDEs and Nonlinear Fluctuation'', where she partially worked on this project. RP and MR received support from the Deutsche Forschungsgemeinschaft (DFG) through grant CRC 1114 ``Scaling Cascades in Complex Systems'', Project C08.

\bibliographystyle{JPE.bst}
\bibliography{library}

\end{document}